\documentclass[11pt,a4paper]{amsart}
\usepackage{graphicx} 
\usepackage{enumerate}
\usepackage{hyperref}
\usepackage{fullpage}
\usepackage{xcolor}
\usepackage{amssymb}
\usepackage{tikz}
\usetikzlibrary{positioning}
\usepackage{cleveref}
\usepackage{mathtools}
\usepackage{appendix}

\title{Linear theories of global fields with absolute values}

\author{Arno Fehm}
\author{Pierre Touchard}
\address{Institut f\"{u}r Algebra, Technische Universit\"{a}t Dresden, 01062 Dresden, Germany}
\email{arno.fehm@tu-dresden.de}
\email{pierre.touchard@tu-dresden.de}

\theoremstyle{plain}
\newtheorem{theorem}{Theorem}[section]
\newtheorem{corollary}[theorem]{Corollary}

\newtheorem{proposition}[theorem]{Proposition}
\newtheorem{lemma}[theorem]{Lemma}

\theoremstyle{remark}
\newtheorem{definition}[theorem]{Definition}
\newtheorem{remark}[theorem]{Remark}

\numberwithin{equation}{section}

\begin{document}

\begin{abstract}
We study the theory of a global field $k$ as a $k$-vector space with a predicate for one of the absolute values on $k$.
For example, we prove that in this language a global field with an ultrametric or real archimedean absolute value has a decidable theory,
while with a complex absolute value the theory is always undecidable.
We also study the existential theories and
axiomatize $k$ together with predicates for all non-complex absolute values on $k$ simultaneously.
\end{abstract}

\maketitle

\section{Introduction}
\label{sec:intro}

\noindent
The first-order theory of a global field (i.e.~a finite extension of $\mathbb{Q}$ or of $\mathbb{F}_p(t)$) in the language of rings is always undecidable \cite{Robinson,Rumely}, 
and the decidability of its existential theory is a famous open problem,
see \cite{Koenigsmann_survey}. 
We therefore investigate the theory of a global field in a language with only (unary) scalar multiplication instead of (binary) multiplication (the so-called {\em linear} theory, see e.g.~\cite{Dries,Weispfenning,Sturm,GHW,BF}) but with additional arithmetic structure. 
Namely, for a field $k$ let $\mathbb{P}_k$ denote the set of {\em places} of $k$, i.e.~equivalence classes of nontrivial absolute values on $k$. 
Following number theory convention, we denote elements of $\mathbb{P}_k$ by letters like $\mathfrak{p}$, and by $|.|_\mathfrak{p}$ a fixed absolute value in the equivalence class. 
A place $\mathfrak{p}$ of a global field is called {\em finite} if $|.|_\mathfrak{p}$ is ultrametric, and {\em infinite} otherwise; in the latter case, $|.|_\mathfrak{p}$ is archimedean, and the completion $\hat{k}_\mathfrak{p}$ 
of $k$ with respect to $|.|_\mathfrak{p}$ (with corresponding place $\hat{\mathfrak{p}}$)
is either isomorphic to $\mathbb{C}$ ($\mathfrak{p}$ is {\em complex}) or $\mathbb{R}$ ($\mathfrak{p}$ is {\em real}), cf.~\cite[Chapter 1]{EP}.
We define on $k$ the relations
\begin{eqnarray*}
    L_\mathfrak{p}(x,y) &:\Longleftrightarrow& |x|_\mathfrak{p}\leq|y|_\mathfrak{p},\\
    M_\mathfrak{p}(x,y,z) &:\Longleftrightarrow& |xy|_\mathfrak{p}=|z|_\mathfrak{p}.
\end{eqnarray*}
Our first result deals with the full or existential theory of one place of a global field and its completion,
in a language with unary function symbols $\cdot_\lambda$ for scalar multiplication with $\lambda\in k$:
\begin{theorem}\label{thm:intro1}
Let $\mathfrak{p}\in\mathbb{P}_k$ be a place of a global field $k$. 
\begin{enumerate}[(a)]
    \item If $\mathfrak{p}$ is finite, then
$$
 {\rm Th}(k,+,0,1,(\cdot_\lambda)_{\lambda\in k},L_\mathfrak{p},M_\mathfrak{p}) = {\rm Th}(\hat{k}_\mathfrak{p},+,0,1,(\cdot_\lambda)_{\lambda\in k},L_{\hat{\mathfrak{p}}},M_{\hat{\mathfrak{p}}}),
$$
and this theory is decidable.
\item If $\mathfrak{p}$ is real, then 
$$
 {\rm Th}(k,+,0,1,(\cdot_\lambda)_{\lambda\in k},L_\mathfrak{p}) = {\rm Th}(\hat{k}_\mathfrak{p},+,0,1,(\cdot_\lambda)_{\lambda\in k},L_{\hat{\mathfrak{p}}}),
$$
and this theory is is decidable.
However,
${\rm Th}(k,+,0,1,(\cdot_\lambda)_{\lambda\in k},L_\mathfrak{p},M_\mathfrak{p})$ is undecidable, in particular
$$
 {\rm Th}(k,+,0,1,(\cdot_\lambda)_{\lambda\in k},L_\mathfrak{p},M_\mathfrak{p}) \neq {\rm Th}(\hat{k}_\mathfrak{p},+,0,1,(\cdot_\lambda)_{\lambda\in k},L_{\hat{\mathfrak{p}}},M_{\hat{\mathfrak{p}}}),
$$
and ${\rm Th}_\exists(k,+,0,1,(\cdot_\lambda)_{\lambda\in k},L_\mathfrak{p},M_\mathfrak{p})$ is Turing-equivalent
to ${\rm Th}_\exists(k,+,\cdot)$.
\item If $\mathfrak{p}$ is complex, then
${\rm Th}(k,+,L_\mathfrak{p})$ is undecidable, in particular
$$
 {\rm Th}(k,+,L_\mathfrak{p}) \neq {\rm Th}(\hat{k}_\mathfrak{p},+,L_{\hat{\mathfrak{p}}}).
$$
Moreover, if in addition $k\cap\mathbb{R}i=\{0\}$ (where we view $k$ as a subfield of $\hat{k}_\mathfrak{p}=\mathbb{C}$), then
${\rm Th}_\exists(k,+,1,(\cdot_\lambda)_{\lambda\in k},L_\mathfrak{p})$ is decidable,
and otherwise it is Turing-equivalent to 
${\rm Th}_\exists(k\cap\mathbb{R},+,\cdot)$.
\end{enumerate}
\end{theorem}

Note that as stated above, the decidability of
${\rm Th}_\exists(k,+,\cdot)$
and ${\rm Th}_\exists(k\cap\mathbb{R},+,\cdot)$ is open.
We derive (a) from a quantifier elimination result for valued vector spaces that we prove in Section~\ref{sec:QE}, and which gives us an explicit axiomatization of the theory in question.
In Remark~\ref{rem:other_QE}
we discuss related work,
including a similar result by van den Dries \cite{Dries} that allows the same conclusion without $M_\mathfrak{p}$ respectively $M_{\hat{\mathfrak{p}}}$, and which also gives us most of (b).
For the rest of (b), and for (c), we mainly exploit different ways to define multiplication (in some cases existentially) from addition and the absolute value.

Then we deduce results for (possibly infinite) families of places of global fields:
\begin{theorem}\label{thm:intro2}
Let $k$ be a global field and let $S_1\subseteq S_0\subseteq\mathbb{P}_k$ be sets of places of $k$. Then
$$
 {\rm Th}(k,+,0,1,(\cdot_\lambda)_{\lambda\in k},(L_\mathfrak{p})_{\mathfrak{p}\in S_0},(M_\mathfrak{p})_{\mathfrak{p}\in S_1})
$$
is decidable if and only if
$S_0$ contains no complex places and
$S_1$ contains no infinite places.
\end{theorem}

In the decidable case, we get model completeness (and if $S_1=\emptyset$, again quantifier elimination) together with an explicit axiomatization of this theory (in a suitable expansion of the language). 
The key ingredients are the axiomatizations for the individual $\mathfrak{p}$, as well as an axiomatization of weak approximation, see Section \ref{sec:WA}, where Theorems \ref{thm:intro1} and \ref{thm:intro2} are proven.

\begin{remark}
For example, for the global field $\mathbb{Q}$ with its finite places identified with the prime numbers $\mathbb{P}$, and the infinite place $\infty$ corresponding to the usual absolute value $|.|$, we get that
$$
 {\rm Th}(\mathbb{Q},+,0,1,(\cdot_\lambda)_{\lambda\in\mathbb{Q}},(L_p)_{p\in\mathbb{P}\cup\{\infty\}},(M_p)_{p\in\mathbb{P}})
$$
is decidable. 
Note that 
$xy=z$ if and only if $xyz\geq0$ and $|xy|=|z|$, 
so since $x\geq 0\Leftrightarrow |x-1|\leq|x+1|$,
multiplication can be defined using $M_\infty$ (cf.~Proposition \ref{prop:real_undecidable}).
Moreover, due to the product formula, 
$M_\infty$ in turn can be defined 
by the $\mathcal{L}_{\omega_1,\omega}$-formula
$\bigwedge_{p\in\mathbb{P}}M_p(x,y,z)$. 
\end{remark}

\section{Places and vector spaces}

\noindent
Fix a field $k$.
As in the introduction, $\mathbb{P}_k$ denotes the set of all places of $k$, i.e.~equivalence classes of nontrivial absolute values.
We denote by $\mathbb{P}_k'\subseteq\mathbb{P}_k$ the subset of those places $\mathfrak{p}$ that are either infinite, or finite and {\em discrete}, meaning that $|k^\times|_\mathfrak{p}$ is a discrete subset of $\mathbb{R}_{>0}$.
Note that if $k$ is a global field, then $\mathbb{P}_k'=\mathbb{P}_k$.

Let $\mathcal{L}_k=\{+,0,1,(\cdot_\lambda)_{\lambda\in k}\}$ be the one-sorted language of vector spaces over $k$,
where $+$ is a binary function symbol, $0$ and $1$ are constant symbols, and $\cdot_\lambda$ is a unary function symbol, denoting scalar multiplication by $\lambda$.
Let $T_k$ be the $\mathcal{L}_k$-theory of $k$-vector spaces,
and $\underline{k}=(k,+,0,1,(\cdot_\lambda)_{\lambda\in k})$ the $k$-vector space $k$ as an $\mathcal{L}_k$-structure.
For a language $\mathcal{L}$ and an $\mathcal{L}$-structure $K$ we denote by $\mathcal{L}(K)$ the expansion of $\mathcal{L}$ by constant symbols for the elements of $K$.

\begin{proposition}\label{prop:Tk}
$\underline{k}$ is a model of $T_k$ 
and admits a unique homomorphism to every $K\models T_k$, which is an embedding if and only if $0^K\neq 1^K$.
\end{proposition}

\begin{proof}
Every homomorphism $f\colon \underline{k}\rightarrow K$ must map $1^k$ to $1^K$, and therefore
$\lambda=\lambda\cdot 1^k$ to $\lambda\cdot 1^K$, for every $\lambda\in k$.
On the other hand, $\lambda\mapsto \lambda\cdot 1^K$ always defines a linear map, and therefore an $\mathcal{L}_k$-homomorphism. It is injective (and then automatically an embedding) if and only if $1^K$ is not the zero vector, i.e.~$0^K$.
\end{proof}

\section{Quantifier elimination for discretely valued vector spaces}
\label{sec:QE}

\noindent
We fix a field $k$ and a nontrivial
(normalized) discrete valuation $u\colon k\twoheadrightarrow\mathbb{Z}\cup\{\infty\}$ with residue field $\kappa$,
in the sense of e.g.~\cite[Chapter I]{Serre}.
We remind the reader that there is an overlap of this section with the results in \cite{Dries}, which we will discuss in detail in Remark \ref{rem:other_QE}.

\begin{definition}
Let $\mathcal{L}_{k,{\rm val}}=\mathcal{L}_k\cup\{<,\infty,s;v\}$ be a language 
with two sorts $K$ (the {\em vector space sort}) and $\Gamma$ (the {\em value sort}):
on the vector space sort 
we have the language $\mathcal{L}_k$;
on the value sort
a binary relation symbol $<$,
a constant symbol $\infty$,
and a unary function symbol $s$ ($s$ for successor);
as well as a symbol $v$ for a unary function from the vector space sort to the value sort.
We write $\mathcal{L}_{k,{\rm val}}$-structures as $(K,\Gamma,v)$, sometimes as $(K,v)$,
and $(k,\mathbb{Z}\cup\{\infty\},u)$ or  $(\underline{k},u)$
for $k$ as
an $\mathcal{L}_{k,{\rm val}}$-structure 
in the natural interpretation,
namely with value sort $\Gamma=\mathbb{Z}\cup\{\infty\}$, $v(\lambda)=u(\lambda)$, and $s(n)=n+1$, $s(\infty)=\infty$.
\end{definition}

We write $a,b,\alpha,\beta,x,y,z$ for elements and variables of the vector space sort,
and $\gamma,\delta,\epsilon$ for elements and variables of the value  sort.

\begin{definition}
Let $T_{k,u}$ be the $\mathcal{L}_{k,{\rm val}}$-theory 
consisting of $T_k$ and the following axioms (1)--(7) if $\kappa$ is finite, respectively (1)--(6) and (7') if $\kappa$ is infinite,
for a structure $(K,\Gamma,v)$:
\begin{enumerate}[ (1)]
\item\label{item:vectorspace} $0^K\neq 1^K$ 
\item\label{item:discrete} $(\Gamma,<)$ is a discretely ordered set with largest element $\infty$
\item\label{item:surj} $v$ is surjective
\item\label{item:triangle} $\forall x,y(v(x+y)\geq\min\{v(x),v(y)\}\wedge(v(x)=\infty\leftrightarrow x=0))$

\item\label{item:s} $s(\infty)=\infty\wedge\forall \gamma(\gamma\neq\infty\rightarrow (s(\gamma)>\gamma\wedge\neg\exists \delta(\gamma<\delta\wedge\delta<s(\gamma))))$, i.e.~$s$ is the successor function,
\item\label{item:compatible} $\forall x( v(\lambda x)=s^{u(\lambda)}(v(x)))$, for every $\lambda\in k$
\end{enumerate}
where for $n\in\mathbb{Z}\cup\{\infty\}$, $\gamma=s^n(\delta)$ is to be read as
$s^{-n}(\gamma)=\delta$ if $n<0$, and as $\gamma=\infty$ if $n=\infty$,

\begin{enumerate}[ (1)]
\setcounter{enumi}{6}
\item\label{item:residue_finite}    $\forall x,y(v(x)=v(y)\neq\infty\rightarrow\bigvee_{\omega\in\Omega}v(\omega x+y)>v(x))$,
\end{enumerate}
where $\Omega\subseteq\mathcal{O}_u^\times$ is a set of representatives of $\kappa^\times$, if $\kappa$ is finite, and
\begin{enumerate}[ (1')]
\setcounter{enumi}{6}
\item\label{item:residue_infinite}  $\forall x,y_1,\dots,y_n((\bigwedge_{j=1}^nv(x)\leq v(y_j))\rightarrow\bigvee_{i=0}^n\bigwedge_{j=1}^n v(\omega_ix+y_j)=v(x))$ for every $n\in\mathbb{N}$,
\end{enumerate}    
where $\omega_0,\omega_1,\dots$ are elements of $\mathcal{O}_u^\times$ with pairwise distinct residues,
if $\kappa$ is infinite.
\end{definition}

\begin{proposition}\label{prop:Tku_model}
$(\underline{k},u)\models T_{k,u}$
\end{proposition}

\begin{proof}
It is obvious that $(k,\mathbb{Z}\cup\{\infty\},u)$ satisfies $T_k$ and \eqref{item:vectorspace}--\eqref{item:compatible}. If $\Omega\subseteq\mathcal{O}_u^\times$ is a set of representatives of $\kappa^\times$, and $u(x)=u(y)$, then $u(\omega x+ y)=u(x)+u(\omega +y/x)>u(x)$ for 
$\omega\in\Omega$ the representative of $-y/x$, so if $\kappa$ is finite, then \eqref{item:residue_finite} holds;
and if $\omega_0,\omega_1,\ldots\in\Omega$ are pairwise distinct, and $v(x)\leq v(y_j)$ for $j=1,\dots,n$, then any $\omega_i$ with residue different from $-y_j/x$ satisfies $u(\omega_ix+y_j)=u(x)$, so if $\kappa$ is infinite, then (\ref{item:residue_infinite}') holds (see also \cite[Lemma 3]{Dries}). 
\end{proof}

Roughly speaking, models of $T_{k,u}$ are nontrivial vector spaces over $k$ valued in a discretely ordered set,
with a valuation compatible with $u$,
and such that
the residue groups 
$$
 \{x\in K:v(x)\geq\gamma\}/\{x\in K:v(x)>\gamma\}
$$ 
are either all finite of cardinality $|\kappa|$, or are all infinite.

Note that a substructure of a model of $T_{k,u}$ automatically again satisfies
all these axioms except possibly \eqref{item:surj}.
Also note that $1$ so far plays no role other than naming some nonzero element.

\begin{proposition}\label{prop:Tku_minimal}
$(\underline{k},u)$ embeds uniquely into every $(K,\Gamma,v)\models T_{k,u}$.    
\end{proposition}

\begin{proof}
Every embedding of $(\underline{k},u)$ into $(K,\Gamma,v)$ must send $1^k$ to $1^K$, and therefore $\lambda\in k$ to $\lambda\cdot 1^K$ (cf.~Proposition \ref{prop:Tk}), and this also determines the embedding on the value sort by \eqref{item:surj}.
Let $f\colon k\rightarrow K$, $\lambda\mapsto \lambda\cdot 1^K$ and
$g\colon\mathbb{Z}\cup\{\infty\}\rightarrow\Gamma$, $n\mapsto s^n(v(1^K))$.
Then $f$ is injective by \eqref{item:vectorspace},
$f$ and $g$ are compatible with $v$ by \eqref{item:compatible} (i.e.~$v(f(\lambda))=g(u(\lambda))$ for every $\lambda\in k$),
$f$ is compatible with $+$, $0$, $1$ and $\cdot_\lambda$ by the vector space axioms $T_k$,
and $g$ is by definition compatible with $s$,
and therefore also with $<$ by \eqref{item:discrete} and \eqref{item:s}.
\end{proof}

\begin{theorem}\label{thm:QE}
$T_{k,u}$ has quantifier elimination.    
\end{theorem}

\begin{proof}
Let $\underline{L}=(L,\Gamma,v),\underline{L}'=(L',\Gamma',v')\models T_{k,u}$ with a common $\mathcal{L}_{k,{\rm val}}$-substructure $\underline{K}=(K,\Gamma_0,v_0)$
and let $\alpha\in L\setminus K$.
It suffices to show that if $\underline{L}'$ is sufficiently saturated, then the identity on $\underline{K}$ extends to an embedding of the
$\mathcal{L}_{k,{\rm val}}$-substructure $\underline{K}[\alpha]$ of $\underline{L}$
generated by $\underline{K}$ and $\alpha$ 
into $\underline{L}'$.
Note that we can indeed restrict to $\alpha$ in the vector space sort
as $v$ and $v'$ are surjective by \eqref{item:surj}.
Moreover, $\underline{K}[\alpha]$ has
$K+k\alpha=K\oplus k\alpha$ in the vector space sort and
$\Gamma_0\cup v(K+k\alpha)$ in the value sort.

Before we start the argument note that if $v(\alpha)\notin v_0(K)$,
then $v_0(K)\cap s^{\mathbb{Z}}(v(\alpha))=\emptyset$ by \eqref{item:compatible},
hence for $b\in K$ and $\lambda\in k$ the value
\begin{equation}\label{eqn:v}
v(b+\lambda\alpha) = \min\{v(b),s^{u(\lambda)}(v(\alpha))\}\in v_0(K)\sqcup s^\mathbb{Z}(v(\alpha)) 
\end{equation}
is uniquely determined by the ultrametric triangle inequality \eqref{item:triangle}. 
Therefore, we can assume without loss of generality that $v_0(K)=\Gamma_0$, i.e.~that $(K,\Gamma_0,v_0)$ satisfies also \eqref{item:surj}
and is therefore a model of $T_{k,u}$.
Indeed, for every $\gamma\in\Gamma_0\setminus v_0(K)$ there exist
$\alpha\in L$ and $\alpha'\in L'$ with $v(\alpha)=\gamma=v'(\alpha')$ by \eqref{item:surj},
and $\underline{K}[\alpha]\rightarrow\underline{K}[\alpha']$, $b+\lambda\alpha\mapsto b+\lambda\alpha'$ is an isomorphism by \eqref{eqn:v}.

We distinguish three cases:

\vspace{0.1cm}

\noindent
{\bf Case 1:} $v(\alpha)\notin\Gamma_0$.

Define the sets of quantifier-free $\mathcal{L}_{k,{\rm val}}(K)$-formulas
\begin{eqnarray*}
 p_0(x)&:=&\{x\neq a\;:\;a\in K\},\\
 p_1(x)&:=&\{v(x)<v(c)\;:\;c\in K,v(\alpha)<v(c)\}\cup
 \{v(x)>v(b)\;:\;b\in K,v(\alpha)>v(b)\}.
\end{eqnarray*}
We claim that $p_0\cup p_1$ is consistent with the elementary diagram of $\underline{L}'$.
Indeed, first note that 
it suffices to show that given $b_1,\dots,b_n,c_1,\dots,c_m\in K$, there exist infinitely many elements of $K$ that satisfy the corresponding finitely many formulas,
since then there also exists one which is in addition different from finitely many given $a_1,\dots,a_l\in K$.
We assume that $n>0$ and $m>0$, but the other cases are very similar.
As $\Gamma_0$ is totally ordered, we can assume without loss of generality that $n=m=1$.
Since $v(b_1)<v(\alpha)<v(c_1)$,
and $v(K)\cap s^{\mathbb{Z}}(v(\alpha))=\emptyset$,
we have $v(b_1)<v(\lambda b_1)<v(c_1)$
for every $\lambda\in k$ with $u(\lambda)>0$, 
of which there are infinitely many.

So, assuming $\underline{L}'$ is $|K|^+$-saturated, there exists a realization
$\alpha' \in L'$  of $p_0\cup p_1$ in $\underline{L}'$. 
As $\alpha'$ satisfies $p_0$, we have $\alpha'\notin K$ and therefore $K+k\alpha'=K\oplus k\alpha'$, and so
$f\colon K+k\alpha\rightarrow K+k\alpha'$, $b+\lambda\alpha\mapsto b+\lambda\alpha'$
is a $k$-vector space isomorphism.
Moreover,
$v(K+k\alpha)=v_0(K)\sqcup s^{\mathbb{Z}}(v(\alpha))$
and $v'(K+k\alpha')=v_0(K)\sqcup s^{\mathbb{Z}}(v'(\alpha'))$ by \eqref{eqn:v},
and since $\alpha$ and $\alpha'$ satisfy $p_1$,
and therefore $v(\alpha)$ and $v'(\alpha')$ realize the same cut in $v_0(K)$,
the map
$v_0(K)\sqcup s^{\mathbb{Z}}(v(\alpha))\rightarrow v_0(K)\sqcup s^{\mathbb{Z}}(v'(\alpha'))$ given by $v(b+\lambda\alpha)\mapsto v'(b+\lambda\alpha')$ is a well-defined isomorphism of ordered sets by \eqref{eqn:v},
and by definition compatible with $f$.

\vspace{0.1cm}

\noindent
{\bf Case 2:} $v(\alpha)\in \Gamma_0$ but $v(\tilde{\alpha})\notin \Gamma_0$ for some 
$\tilde{\alpha}=b+\lambda\alpha$, 
$b\in K$, $0\neq\lambda\in k$.

In this case, 
since $K+k\alpha=K+k\tilde{\alpha}$,
we can replace $\alpha$ by $\tilde{\alpha}$ and are back in Case~1.

\vspace{0.1cm}

\noindent
{\bf Case 3:} $v(K+k\alpha)=\Gamma_0$.

Let
\begin{eqnarray*}
    p_2&:=&\{v(x-b)=v(c)\;:\;b,c\in K,c\neq 0,v(\alpha-b)=v(c)\}.
\end{eqnarray*}
We claim that $p_0\cup p_2$ is consistent with the elementary diagram of $\underline{L}'$.
Again it suffices to find,
given $(b_1,c_1),\dots,(b_n,c_n)$, infinitely many $\beta\in K$ with $v(\beta-b_i)=v(c_i)$ for each $i=1,\dots,n$.
Without loss of generality, there is $m\leq n$ such that $c_1,\dots,c_m$ are those $c_i$ of maximal valuation among the $c_i$.
If $\kappa$ is finite, then by
\eqref{item:residue_finite} there exists $\lambda\in \mathcal{O}_u^\times$ such that
$v(\lambda c_1-(\alpha-b_1))>v(c_1)$,
so if we let $\beta=(\mu+\lambda) c_1+b_1$
for any $\mu\in k$ with $u(\mu)>0$, 
then for every $i$,
$$
 v(\beta-b_i)=v((\lambda c_1-(\alpha-b_1))+\mu c_1 +(\alpha-b_i)) = v(\alpha-b_i)=v(c_i)
$$ 
by \eqref{item:triangle} and \eqref{item:compatible}.
If $\kappa$ is infinite, then by (\ref{item:residue_infinite}') there exists $\lambda\in\mathcal{O}_u^\times$ with
$v(\lambda c_1+((\alpha-b_i)-(\alpha-b_1)))=v(c_1)$ for $i=1,\dots,m$, so if we again let
$\beta=(\mu+\lambda)c_1+b_1$ for any $\mu\in k$ with $u(\mu)>0$,
then, again by \eqref{item:triangle} and \eqref{item:compatible}, for $i=1,\dots,m$ we have
$$
 v(\beta-b_i)=v((\lambda c_1+((\alpha-b_i)-(\alpha-b_1)))+\mu c_1)=v(\lambda c_1)=v(c_i),
$$
and for $i=m+1,\dots,n$ we get
$$
 v(\beta-b_i)=v((\alpha-b_i)-(\alpha-b_1)+\lambda c_1+\mu c_1)=v(\alpha-b_i)=v(c_i).
$$
So, assuming $\underline{L}'$ is $|K|^+$-saturated, there exists a realization
$\alpha' \in L'$  of $p_0\cup p_2$ in $\underline{L}'$. 
Again we have $\alpha'\notin K$ since it satisfies $p_0$, and so we get an isomorphism $f$ of $k$-vector spaces as in Case~1. 
Since $\alpha'$ satisfies $p_2$, we have
$$
 v(b+\lambda\alpha)=s^{u(\lambda)}(v(\alpha-(-\lambda^{-1}b)))=s^{u(\lambda)}(v'(\alpha'-(-\lambda^{-1}b)))=v'(b+\lambda\alpha')
$$
for every $b\in K$ and $\lambda\in k^\times$ by \eqref{item:compatible},
so in particular $v'(K+k\alpha')=\Gamma_0$,
and $f$ is compatible with the identity on the value sort, hence it gives an embedding of $\underline{K}[\alpha]$ into $\underline{L}'$.
\end{proof}

\begin{definition}
Let $\mathcal{L}_{k,{\rm val}}^+=\mathcal{L}_{k,{\rm val}}\cup\{0,1,+,(P_n)_{n\in\mathbb{N}}\}$ be the expansion of $\mathcal{L}_{k,{\rm val}}$, where 
to the value sort we add constant symbols $0,1$, a binary function symbol $+$, and a family of unary predicates $(P_n)_{n\in\mathbb{N}}$.
Again we write $(\underline{k},u)$ for $k$ as an $\mathcal{L}_{k,{\rm val}}^+$-structure, with $P_n$ interpreted by $n\mathbb{Z}\cup\{\infty\}$, when there is no danger of confusion.
\end{definition}

\begin{definition}
Let $T_{k,u}^+$ be the $\mathcal{L}_{k,{\rm val}}^+$-theory extending $T_{k,u}$ by the following axioms for a structure $(K,\Gamma,v)$:
\begin{enumerate}[ (1)]
\setcounter{enumi}{7}
\item\label{item:1} $v(1)=0$ and $s(0)=1$
\item\label{item:oag} $(\Gamma\setminus\{\infty\},+,0,<)$ is an ordered abelian group, and $\forall \gamma(\gamma+\infty=\infty+\gamma=\infty)$
\item\label{item:Zgroup} $\forall \gamma\exists\delta(\bigvee_{i=1}^n\gamma=n\delta+i1)$, for every $n\in\mathbb{N}$
\item\label{item:Pn} $\forall\gamma( P_n(\gamma)\leftrightarrow \exists \delta(\gamma=n\delta))$, for every $n\in\mathbb{N}$
\end{enumerate}
where as usual we write $n\delta$ for $\delta+\dots+\delta$.
\end{definition}
Note that \eqref{item:oag}+\eqref{item:Zgroup} is equivalent  to $(\Gamma\setminus\{\infty\},+,0,1,<)$ being a $\mathbb{Z}$-group, i.e.~elementarily equivalent to the ordered abelian group $(\mathbb{Z},+,0,1,<)$ \cite[Theorem 4.1.3]{PD}.

\begin{theorem}\label{cor:QE}
$T_{k,u}^+$ has quantifier elimination.
\end{theorem}

\begin{proof}
This follows from Theorem \ref{thm:QE} for formal reasons, explained in Appendix A of \cite{Rid}:
First of all, $\mathcal{L}_{k,{\rm val}}^+$ is a {\em $\{\Gamma\}$-enrichment} of $\mathcal{L}_{k,{\rm val}}$ in the sense of \cite[Definition A.2]{Rid},
and $\Gamma$ is a {\em closed sort} in the sense of \cite[Definition A.7]{Rid}.
By Theorem \ref{thm:QE}, $T_{k,u}$ eliminates quantifiers, in particular it eliminates quantifiers relatively to $\{\Gamma\}$ in the sense of \cite[Definition A.5]{Rid}.
So since $\Gamma$ is closed,
\cite[Proposition A.9]{Rid} gives that 
$T_{k,u}$ eliminates quantifiers {\em resplendently} relatively to $\{\Gamma\}$: any theory in any $\{\Gamma\}$-enrichment of $\mathcal{L}_{k,{\rm val}}$ eliminates quantifiers relative to $\{\Gamma\}$ and the possibly new sorts \cite[Definition A.5]{Rid}.
In particular, $T_{k,u}^+$ eliminates quantifiers relative to $\{\Gamma\}$.
Since the theory of $\mathbb{Z}$-groups in the language $\{+,0,1,<,(P_n)_{n\in\mathbb{N}})\}$ 
with interpretations of the $P_n$ given by \eqref{item:Pn}
has quantifier elimination
\cite[proof of Theorem 4.1.3]{PD},
$T_{k,u}^+$ in fact eliminates all quantifiers.
\end{proof}

\begin{corollary}\label{cor:model_complete}
$T_{k,u}$ and $T_{k,u}^+$ are complete and model complete.    
\end{corollary}

\begin{proof}
Model completeness follows from quantifier elimination, and this implies completeness since
$(\underline{k},u)$ is a minimal model in both cases:
For $T_{k,u}$ these are Propositions \ref{prop:Tku_model} and \ref{prop:Tku_minimal},
and for $T_{k,u}^+$ see the following Proposition \ref{prop:k_models_Tkuplus}.
\end{proof}

\begin{proposition}\label{prop:k_models_Tkuplus}
$(\underline{k},u)\models T_{k,u}^+$, and $(\underline{k},u)$ embeds uniquely into every $(K,\Gamma,v)\models T_{k,u}^+$.   
\end{proposition}

\begin{proof}
As it is clear that $(\underline{k},u)$ satisfies \eqref{item:1}--\eqref{item:Pn} (with $P_n$ interpreted as $n\mathbb{Z}\cup\{\infty\})$, $(\underline{k},u)$ is also a model of $T_{k,u}^+$.
And if $(K,\Gamma,v)\models T_{k,u}^+$, then the maps $f,g$ from the proof of Proposition~\ref{prop:Tku_minimal} are also compatible with the new symbols on the value sort: For $0,1$ this follows from \eqref{item:1}, and
$g(n+m)=s^{n+m}(v(1))=s^n(s^m(v(1)))=s^n(g(m))=g(m)+g(n)$ since $\gamma+s(\delta)=s(\gamma+\delta)$ by \eqref{item:1}+\eqref{item:oag}.
This implies that if $m\in n\mathbb{Z}$, then $g(m)\in n\Gamma$; 
conversely, if $m\in i+n\mathbb{Z}$ for $i=1,\dots,n-1$, then $g(m)\notin n\Gamma$, since otherwise $i1\in n\Gamma$, contradicting 
\eqref{item:s}+\eqref{item:1}.
So since $(K,\Gamma,v)$ satisfies \eqref{item:Pn}, $g$ is compatible with the $P_n$.
\end{proof}

\begin{remark}\label{rem:other_QE}
There is a significant overlap of the results in this section with \S2 of \cite{Dries}, which we discovered after writing this note, and which we will discuss now.
Instead of demanding that $u(k^\times)=\mathbb{Z}$ and that the value set $(\Gamma,<)$ of $K$ is discretely ordered with successor function $s$, van den Dries works with an arbitrary valued field $(k,u)$ and an action $u(k^\times)\times\Gamma\rightarrow\Gamma$. In \cite[Theorem 3]{Dries} he then proves quantifier elimination in the special case where $u(k^\times)$ (and thus also $\Gamma$) has a smallest positive element (which includes the case $u(k^\times)=\mathbb{Z}$ that we consider),
but in spite of the two-sorted algebraic setup,
he does this in a {\em one-sorted} first-order language with a divisibility relation $|$, similar to what we will deduce in Section \ref{sec:ultrametric} below.
In particular, he does not consider a language with a symbol for (binary) addition in the value set (like in our Theorem~\ref{cor:QE}), 
and it seems more difficult to deduce such results in the one-sorted setting,
which is why we decided to not remove our presentation.
We remark that in \cite{Dries}, the analogue of (\ref{item:residue_infinite}'), which is relevant if $\kappa$ is infinite, does not appear as an axiom but is deduced in \cite[Lemma 3]{Dries} from the other axioms.

Other related results on valued vector spaces and valued abelian groups do not cover 
Theorems~\ref{thm:QE} and \ref{cor:QE}:
\cite{KK,Touchard} work only with value-preserving scalar multiplication, \cite{Weispfenning} eliminates quantifiers in discretely valued vector spaces modulo a background theory of valued {\em fields} (so with multiplication in the language),
and \cite{Schmitt,Mariaule} obtain quantifier elimination for valued abelian groups like $\mathbb{Z}_p$, but do not apply to $\mathbb{Q}_p$.
\end{remark}

\section{Absolute values}
\noindent
In order to treat ultrametric and archimedean absolute values on $k$ simultaneously, we now switch to the one-sorted language
$\mathcal{L}_{k}$
that we expand by binary respectively ternary relation symbols $L$ and $M$ as in the introduction.
To make formulas more readable we will use the abbreviations
\begin{eqnarray*}
  L(x,y,z) &:\Longleftrightarrow& L(x,y)\wedge L(y,z) \\
  \Lambda(x,y) &:\Longleftrightarrow& L(x,y)\wedge L(y,x)\\
  \Lambda(x,y,z) &:\Longleftrightarrow& L(x,y)\wedge L(y,z)\wedge L(z,x).
\end{eqnarray*}
 
\subsection{Ultrametric absolute values}\label{sec:ultrametric}
An ultrametric absolute value $|.|$ gives rise to a valuation (of rank 1) given by $v(x)=-\log|x|$, and this induces a bijection between the set of finite places of $k$ and the equivalence classes of nontrivial valuations on $k$ of rank $1$.

\begin{definition}
We denote by 
$\mathcal{L}_{k,{\rm fin}}$ 
the one-sorted language $\mathcal{L}_k\cup\{L\}$
where $L$ is a binary relation symbol,
and by
$\mathcal{L}_{k,{\rm fin}}^+$ the one-sorted language $\mathcal{L}_{k,{\rm fin}}\cup\{M,(Q_n)_{n\in\mathbb{N}}\}$,
where $M$ is a ternary and $Q_n$ is a unary relation symbol.
For $\mathfrak{p}\in\mathbb{P}_k$,
we write $(\underline{k},\mathfrak{p})$ for the
$\mathcal{L}_{k,{\rm fin}}^+$-structure
$(\underline{k},L_\mathfrak{p},M_\mathfrak{p},(Q_{n,\mathfrak{p}})_{n\in\mathbb{N}})$ with $L_\mathfrak{p}$ and $M_\mathfrak{p}$ as defined in the introduction,
and 
$$
 Q_{n,\mathfrak{p}}=\left\{x\in k:\exists y\in k(|x|_\mathfrak{p}=|y|_\mathfrak{p}^n)\right\}.
$$
\end{definition}

From now on we will sometimes state results for $\mathcal{L}_{k,{\rm val}}$ and $\mathcal{L}_{k,{\rm val}}^+$
or for
$\mathcal{L}_{k,{\rm fin}}$ and $\mathcal{L}_{k,{\rm fin}}^+$
simultaneously, and we will indicate this by writing 
$\mathcal{L}_{k,{\rm val}}^{[+]}$ respectively
$\mathcal{L}_{k,{\rm fin}}^{[+]}$, and then similarly for theories.

In the following,
we work with interpretations in the notation of \cite[\S5.3]{Hodges_long}.
In particular, $\delta_\varphi$ denotes the domain formula and $f_\varphi$ denotes the coordinate map of the interpretation $\varphi$.

\begin{lemma}\label{lem:interpretation}
Let $\pi$ be a uniformizer of $u$, i.e.~$u(\pi)=1$.
\begin{enumerate}[(a)]
    \item Every $\mathcal{L}_{k,{\rm val}}^{[+]}$-structure $K=(K,\Gamma,v)$ interprets an $\mathcal{L}_{k,{\rm fin}}^{[+]}$-structure $K^{(1)}$ by setting $\delta_\varphi(x)=x\in K$, $f_\varphi={\rm id}_K$,
\begin{eqnarray*}
\label{eqn:L}    L(x,y)_\varphi &:=& v(x)\geq v(y), \\
\label{eqn:M} M(x,y,z)_\varphi &:=& v(x)+v(y)=v(z),\\
\label{eqn:Q} Q_{n}(x)_\varphi &:=& P_n(v(x)),
\end{eqnarray*}    
where the last two lines are to be ignored
in case of $\mathcal{L}_{k,{\rm val}}$.
In particular, there exists a reduction map $\varphi\mapsto\varphi_\varphi$ from $\mathcal{L}_{k,{\rm fin}}^{[+]}$-formulas to $\mathcal{L}_{k,{\rm val}}^{[+]}$-formulas such that for every $\mathcal{L}_{k,{\rm val}}^{[+]}$-structure $K=(K,\Gamma,v)$ 
and every $a\in K^n$
we have $K^{(1)}\models\varphi(a)\Leftrightarrow K\models\varphi_\varphi(a)$
(where we identify the domain of $K^{(1)}$ with the vector space sort of $K$).
\item Every $\mathcal{L}_{k,{\rm fin}}^{[+]}$-structure $K$ 
that satisfies the admissibility conditions $T_\Psi^{[+]}$ (see \cite[p.~214]{Hodges_long}) of the
domain formula
\begin{eqnarray*}
    \delta_\Psi(x_1,x_2)&:=&x_2=0\vee x_2=1
\end{eqnarray*}
and the defining formulas
\begin{eqnarray*}
    (x=y)_\Psi &:=& x_2=0\wedge y_2=0\wedge x_1=y_1,\\
    (x=0)_\Psi &:=& x_2=0\wedge x_1=0,\\
    (x=1)_\Psi &:=& x_2=0\wedge x_1=1,\\
    (x+y=z)_\Psi &:=& x_2=0\wedge y_2=0\wedge z_2=0\wedge x_1+y_1=z_1,\\
    (y=\lambda x)_\Psi &:=& x_2=0\wedge y_2=0\wedge y_1=\lambda x_1,\\
    (\gamma=\delta)_\Psi &:=& x_2=1\wedge y_2=1\wedge \Lambda(x_1,y_1),\\
    (\gamma=0)_\Psi &:=& x_2=1\wedge \Lambda(x_1,1) ,\\
    (\gamma=1)_\Psi &:=& x_2=1\wedge \Lambda(x_1,\pi),\\
    (\gamma=\infty)_\Psi &:=& x_2=1\wedge x_1=0,\\
    (\gamma\leq\delta)_\Psi &:=& x_2=1\wedge y_2=1\wedge L(x_1,y_1) ,\\
    (s(\gamma)=\delta)_\Psi &:=& x_2=1\wedge y_2=1\wedge \Lambda(\pi x_1,y_1),\\
    (v(x)=\gamma)_\Psi &:=& x_2=0\wedge y_2=1\wedge \Lambda(x_1,y_1),\\
        (\gamma=\delta+\epsilon)_\Psi &:=& x_2=1\wedge y_2=1\wedge z_2=1\wedge M(x_1,y_1,z_1),\\
    (P_n(\gamma))_\Psi &:=& x_2=1\wedge Q_n(x_1),
\end{eqnarray*}
where again the last two lines are to be ignored in case $\mathcal{L}_{k,{\rm fin}}$,
interprets an $\mathcal{L}_{k,{\rm val}}^{[+]}$-structure $K^{(2)}=(K,K/\sim,v)$ by setting 
\begin{eqnarray*}
 x\sim y &:\Leftrightarrow& K\models L(x,y)\wedge L(y,x)\\
 v(x) &:=& [x]_\sim := \{y\in K : x\sim y\}\\
    f_\Psi(x_1,x_2)&:=&\begin{cases}x_1,&\mbox{if }x_2=0,\\
    [x_1]_{\sim},&\mbox{if }x_2=1.\end{cases}    
\end{eqnarray*}
In particular, there exists a reduction map $\varphi\mapsto\varphi_\Psi$ from $\mathcal{L}_{k,{\rm val}}^{[+]}$-formulas 
to $\mathcal{L}_{k,{\rm fin}}^{[+]}$-formulas,
such that if $\varphi$ has no free variables in the value sort,
and $K$  is an $\mathcal{L}_{k,{\rm fin}}^{[+]}$-structure that satisfies $T_\Psi^{[+]}$,
we have $K^{(2)}\models\varphi(a)\Leftrightarrow K\models\varphi_\Psi'(a)$ for every $a\in K^n$, where
$$
 \varphi_\Psi'(x_1,\dots,x_n):=\varphi_\Psi(x_1,0,\dots,x_n,0).
$$ 
\item For every $\mathcal{L}_{k,{\rm val}}^{[+]}$-structure $K=(K,\Gamma,v)$, we have $(K^{(1)})^{(2)}=(K,v(\Gamma),v)$, so if $K$ satisfies \eqref{item:surj}, then $(K^{(1)})^{(2)}\cong K$.
If $K$ is an $\mathcal{L}_{k,{\rm fin}}^{[+]}$-structure that satisfies $T_\Psi^{[+]}$, then $(K^{(2)})^{(1)}\cong K$. 
In fact, $\varphi$ and $\Psi$ give biinterpretations in these cases.
\item 
If $\varphi$ is existential, then so is $\varphi_\Psi'$,
and if $\varphi$ is a quantifier-free $\mathcal{L}_{k,{\rm val}}$-formula, then $\varphi_\Psi'$ is a quantifier-free $\mathcal{L}_{k,{\rm fin}}$-formula.
\end{enumerate}    
\end{lemma}

\begin{proof}
(a) and (b) are clear, and the isomorphisms $(K^{(1)})^{(2)}\cong K$ respectively $(K^{(2)})^{(1)}\cong K$ are easily checked.
Homotopies between the identity and $\varphi\circ\Psi$, and between the identity and $\Psi\circ\varphi$ (cf.~\cite[\S5.4(c)]{Hodges_long}) are given by the formulas 
$$
 \chi_{\varphi\circ\Psi}(x;y_1,y_2) \;:=\; x=y_1\wedge y_2=0.
$$
respectively
$$
 \chi_{\Psi\circ\varphi}(\xi;\eta_1,\eta_2) \;:=\; \eta_1\in K\wedge\eta_2\in K\wedge((\xi\in K\wedge \xi=\eta_1\wedge \eta_2=0)
 \vee (\xi\in\Gamma\wedge \xi=v(\eta_1)\wedge \eta_2=1)),
$$
where $\xi,\eta_1,\eta_2$ denote variables in any of the two sorts.

(d): The reduction map of $\Psi$ preserves existential formulas  simply since $\delta_\Psi$ and all the defining formulas of $\Psi$ are existential, in fact even quantifier-free (cf.~\cite[Remark 3 on p.~215]{Hodges_long}).
The fact that the reduction map can be made to preserve quantifier-free formulas follows from the observation that in $\mathcal{L}_{k,{\rm val}}$, 
every term in the vector space sort can be translated verbatim to a $\mathcal{L}_{k,{\rm fin}}$-term, and
every term in the value sort is of the form $s^n(\gamma)$ for $n\in\mathbb{Z}$ and $\gamma$ either a variable or $\infty$, which can be translated directly to $\pi^nx$ respectively $0$, without introducing additional existential quantifiers, e.g.
\begin{eqnarray*}
 (v(x+y)=s^n(\gamma))_\Psi &=& x_2=0\wedge y_2=0\wedge z_2=1\wedge \Lambda(x_1+y_1,\pi^nz_1).\qedhere
\end{eqnarray*}
\end{proof}

\begin{definition}
For $\mathfrak{p}\in\mathbb{P}_k'$ finite,
we let $T_{k,\mathfrak{p}}^{[+]}$ be the $\mathcal{L}_{k,{\rm fin}}^{[+]}$-theory 
$T_k\cup T_\Psi^{[+]}\cup\{\varphi_\Psi:\varphi\in T_{k,v_\mathfrak{p}}^{[+]}\}$.
\end{definition}

\begin{theorem}\label{thm:finite}
For $\mathfrak{p}\in\mathbb{P}_k'$ finite,
$T_{k,\mathfrak{p}}$ is complete and has quantifier elimination, 
$T_{k,\mathfrak{p}}^+$ is complete and model complete,
and $(\underline{k},\mathfrak{p})\models T_{k,\mathfrak{p}}^{[+]}$ embeds uniquely into every model of $T_{k,\mathfrak{p}}^{[+]}$.
\end{theorem}

\begin{proof}
Note that every $K\models T_{k,\mathfrak{p}}^{[+]}$ is biinterpretable with $K^{(2)}\models T_{k,v_\mathfrak{p}}^{[+]}$ (Lemma \ref{lem:interpretation}(c)),
so completeness of $T_{k,\mathfrak{p}}^{[+]}$ follows from completeness of $T_{k,v_\mathfrak{p}}^{[+]}$ (Corollary \ref{cor:model_complete}).
Model completeness of $T_{k,\mathfrak{p}}^+$ follows from the fact that $T_{k,v_\mathfrak{p}}^+$ is model complete (Corollary \ref{cor:model_complete}) and that the reduction map of $\Psi$ preserves existential formulas (Lemma \ref{lem:interpretation}(d)).
Similarly, quantifier elimination of 
 $T_{k,\mathfrak{p}}$ follows from the fact that $T_{k,v_\mathfrak{p}}$ has quantifier elimination (Theorem \ref{thm:QE}) and that the reduction map of $\Psi$ preserves quantifier-free formulas when restricted to $\mathcal{L}_{k,{\rm val}}$ (Lemma \ref{lem:interpretation}(d)).
 Since $(\underline{k},\mathfrak{p})^{(2)}\cong(\underline{k},v_\mathfrak{p})\models T_{k,v_\mathfrak{p}}^+$ (Proposition \ref{prop:k_models_Tkuplus}), we see that
 $(\underline{k},\mathfrak{p})\models T_{k,\mathfrak{p}}^+$.
 Since there is a unique embedding of $\underline{k}$ into every model of $T_k$ that satisfies \eqref{item:vectorspace} (Proposition \ref{prop:Tk}), there is at most one embedding of $(\underline{k},\mathfrak{p})$ into a fixed $K\models T_{k,\mathfrak{p}}^{[+]}$,
 and since the unique $\mathcal{L}_k$-embedding $f\colon\underline{k}\rightarrow K$ is 
 compatible with the unique $\mathcal{L}_{k,{\rm val}}^{[+]}$-embedding 
 $(\underline{k},\mathfrak{p})^{(2)}\cong(\underline{k},v_\mathfrak{p})\rightarrow K^{(2)}$ (Proposition \ref{prop:k_models_Tkuplus}),
 we see that $f$ is in fact an $\mathcal{L}_{k,\mathfrak{p}}^{[+]}$-embedding.
\end{proof}

\begin{remark}\label{rmk:finiteplaceL+noEQ}
In the situation of Theorem \ref{thm:finite}, $T^+_{(k,\mathfrak{p})}$ does not have quantifier elimination.
For example, the $\mathcal{L}_{k,{\rm fin}}^+$-formula 
\[ 
\varphi(x,y) \coloneqq \exists z   (M(x,x,z) \wedge M(x,z,y)),
\] 
stating $|x^3|_\mathfrak{p}=|y|_\mathfrak{p}$, is not equivalent modulo $T^+_{(k,\mathfrak{p})}$ to a quantifier-free formula.  
To see this, 
use the compactness theorem to obtain an elementary extension $(K,\Gamma,v)$ of the valued {\em field}
$(k,\mathbb{Z}\cup \{\infty\}, v_\mathfrak{p})$ that contains an element $0\neq a \in K$ with $v(a)\in n\Gamma$ for every $n \in \mathbb{N}$. 
One can then show
that the
$\mathcal{L}_{k,{\rm fin}}^{+}$-substructures
$k \oplus ak \oplus a^3 k$ and $k \oplus ak \oplus a^4 k$
of $(K|_{\mathcal{L}_{k,{\rm val}}^+})^{(1)}$ are isomorphic via the $k$-linear map sending the basis $(1,a,a^3)$ to $(1,a,a^4)$,
although $(K,\Gamma,v)\models\varphi(a,a^3)\wedge\neg\varphi(a,a^4)$.
\end{remark}

\subsection{Real absolute values}

A real place $\mathfrak{p}$ of $k$ defines an embedding $k\rightarrow\mathbb{R}=\hat{k}_\mathfrak{p}$ and therefore an (archimedean) ordering $\leq_\mathfrak{p}$ on $k$.

\begin{lemma}\label{lem:real}
For a real place $\mathfrak{p}$ of $k$, and $x,y\in k$, we have
\begin{eqnarray*}
\label{eqn:leq} x\leq_\mathfrak{p} y &\Longleftrightarrow& |y-x-1|_\mathfrak{p}\leq|y-x+1|_\mathfrak{p},\\
 |x|_\mathfrak{p}\leq|y|_\mathfrak{p} &\Longleftrightarrow& 
 y\geq_\mathfrak{p}x\geq_\mathfrak{p}0\vee y\geq_\mathfrak{p}-x\geq_\mathfrak{p}0\vee -y\geq_\mathfrak{p}x\geq_\mathfrak{p}0\vee -y\geq_\mathfrak{p}-x\geq_\mathfrak{p}0.
\end{eqnarray*}
\end{lemma}

\begin{proof}
The first equivalence holds in $\mathbb{R}$ since the Möbius transformation $z\mapsto\frac{z-1}{z+1}$ maps $\mathbb{R}_{\geq0}$ to $[-1,1)$ and $\infty$ to $1$,
and therefore it holds in $k$.
The second equivalence is obvious.
\end{proof}

\begin{proposition}\label{prop:real_undecidable}
For a real place $\mathfrak{p}$ of $k$,
multiplication on $k$ is quantifier-freely definable in
the structure $(k,+,0,1,\cdot_{-1},L_\mathfrak{p},M_\mathfrak{p})$,
uniformly in $k$. 
\end{proposition}

\begin{proof}
For $a,b,c\in k$ we have that
$ab=c$ if and only if $|ab|_\mathfrak{p}=|c|_\mathfrak{p}$ and $abc\geq_\mathfrak{p}0$. 
By Lemma~\ref{lem:real}, $\varphi(x):=L(x-1,x+1)$ holds if and only if $x\geq_\mathfrak{p} 0$,
so, with $\psi(x,y,z):=\varphi(x)\wedge\varphi(y)\wedge\varphi(z)$, the formula
$$
 M(x,y,z)\wedge (\psi(x,y,z)\vee\psi(-x,-y,z)\vee\psi(-x,y,-z)\vee\psi(x,-y,-z))
$$
defines multiplication.
\end{proof}

\begin{corollary}\label{cor:real_undecidable}
If $\mathfrak{p}$ is a real place of a number field $k$, then ${\rm Th}(k,+,0,1,\cdot_{-1},L_\mathfrak{p},M_\mathfrak{p})$ is undecidable, 
and ${\rm Th}_\exists(k,+,0,1,\cdot_{-1},L_\mathfrak{p},M_\mathfrak{p})$,
${\rm Th}_\exists(k,+,0,1,(\cdot_\lambda)_{\lambda\in k},L_\mathfrak{p},M_\mathfrak{p})$
and ${\rm Th}_\exists(k,+,\cdot)$
are Turing-equivalent.
\end{corollary}

\begin{proof}
As $(k,+,\cdot)$ is quantifier-freely definable in $(k,+,0,1,\cdot_{-1},L_\mathfrak{p},M_\mathfrak{p})$
by Proposition \ref{prop:real_undecidable},
we get Turing reductions from
${\rm Th}(k,+,\cdot)$ to
${\rm Th}(k,+,0,1,\cdot_{-1},L_\mathfrak{p},M_\mathfrak{p})$
and from
${\rm Th}_\exists(k,+,\cdot)$ to
${\rm Th}_\exists(k,+,0,1,\cdot_{-1},L_\mathfrak{p},M_\mathfrak{p})$.
In particular, since ${\rm Th}(k,+,\cdot)$ is undecidable (see Section \ref{sec:intro}), 
${\rm Th}(k,+,0,1,\cdot_{-1},L_\mathfrak{p},M_\mathfrak{p})$ is undecidable.
Finally, 
as $\leq_\mathfrak{p}$ is existentially definable in $(k,+,\cdot)$ \cite[Lemma 4.1]{EM},
also $L_\mathfrak{p}$ and $M_\mathfrak{p}$
are existentially definable in $(k,+,\cdot)$,
and so there is a quantifier-free interpretation of 
$(k,+,0,1,(\cdot_\lambda)_{\lambda\in k},L_\mathfrak{p},M_\mathfrak{p})$
in $(k,+,\cdot)$,
which gives the Turing reduction from
${\rm Th}_\exists(k,+,0,1,(\cdot_\lambda)_{\lambda\in k},L_\mathfrak{p},M_\mathfrak{p})$ to
${\rm Th}_\exists(k,+,\cdot,(\lambda)_{\lambda\in k})$, the existential theory in the language of fields with constant symbols for the elements of $k$,
and we can existentially quantify away the constants.
\end{proof}

\begin{definition}
We let $\mathcal{L}_{k,{\rm ord}}$ be the one-sorted language $\mathcal{L}_{k}\cup\{<\}$
where $<$ is a binary relation symbol.
We let $\mathcal{L}_{k,{\rm inf}}=\mathcal{L}_{k,{\rm fin}}$ be the one-sorted language $\mathcal{L}_{k}\cup\{L\}$,
and $\mathcal{L}_{k,{\rm inf}}^+$
the one-sorted language $\mathcal{L}_{k}\cup\{L,M\}$.
\end{definition}

\begin{definition}\label{def:Tkleq}
For an ordering $\leq'$ on $k$,
$T_{k,\leq'}$ is the $\mathcal{L}_{k,{\rm ord}}$-theory $T_k$ together with the following axioms for an $\mathcal{L}_{k,{\rm ord}}$-structure $(K,+,0,1,(\cdot_\lambda)_{\lambda\in k},\leq)$:
\begin{enumerate}
\item $0<1$
\item $(K,0,+,<)$ is an ordered abelian group
\item $\forall x,y(x\leq y\rightarrow \lambda x\leq\lambda y)$, for all $\lambda\in k$ with $\lambda\geq'0$.
\end{enumerate}
\end{definition}

\begin{theorem}\label{thm:QE_order}
The theory $T_{k,\leq'}$ has quantifier elimination and is complete.
\end{theorem}

\begin{proof}
See \cite[Theorem 1]{Dries}.
\end{proof}

\begin{lemma}\label{lem:interpretation_order}
\begin{enumerate}[(a)]
    \item Every $\mathcal{L}_{k,{\rm ord}}$-structure $(K,\leq)$ interprets an $\mathcal{L}_{k,{\rm inf}}$-structure $(K,\leq)^L$ by the defining formula 
\begin{eqnarray*}
    L(x,y)_\varphi &:=& y\geq x\geq 0\vee y\geq -x\geq 0\vee -y\geq x\geq 0\vee -y\geq -x\geq 0.
\end{eqnarray*}
In particular, there is a reduction map $\psi\mapsto\psi_\varphi$ from $\mathcal{L}_{k,{\rm inf}}$-formulas to $\mathcal{L}_{k,{\rm ord}}$-formulas such that, for every $(K,\leq)$ and $a\in K^n$, $(K,\leq)^L\models\psi(a)$ if and only if $(K,\leq)\models\psi_\varphi(a)$.
\item Every $\mathcal{L}_{k,{\rm fin}}$-structure $(K,L)$
interprets an $\mathcal{L}_{k,{\rm ord}}$-structure $(K,L)^\leq$ by the defining formula
\begin{eqnarray*}
(x\leq y)_\Psi &:=& L(y-x-1,y-x+1).
\end{eqnarray*}
In particular, there is a reduction map $\varphi\mapsto\varphi_\Psi$ from $\mathcal{L}_{k,{\rm ord}}$-formulas to $\mathcal{L}_{k,{\rm inf}}$-formulas such that, for every $(K,L)$ and $a\in K^n$, $(K,L)^\leq\models\varphi(a)$ if and only if $(K,L)\models\varphi_\Psi(a)$.
\item The reduction maps $\psi\mapsto\psi_\varphi$ and $\varphi\mapsto\varphi_\Psi$ preserve quantifier-free formulas.
\end{enumerate}    
\end{lemma}

\begin{proof}
(a) and (b) are trivial. 
The reduction maps preserve quantifier-free formulas since the defining formulas are quantifier-free and there are no additional constant or function symbols to interpret.
\end{proof}

\begin{definition}
For a real place $\mathfrak{p}$ of $k$,
$T_{k,\mathfrak{p}}$ is the $\mathcal{L}_{k,{\rm inf}}$-theory 
$$
 T_k\cup\{\varphi_\Psi:\varphi\in T_{k,\leq_\mathfrak{p}}\}\cup\{\forall x,y(L(x,y)\leftrightarrow L(x,y)_{\varphi\Psi})\}.
$$ 
\end{definition}

\begin{theorem}\label{thm:real}
For a real place $\mathfrak{p}$ of $k$, $T_{k,\mathfrak{p}}$ has quantifier elimination,
is complete and model complete,
and $(\underline{k},L_\mathfrak{p})\models T_{k,\mathfrak{p}}$ embeds uniquely into every $K\models T_{k,\mathfrak{p}}$, uniformly in $k$.
\end{theorem}

\begin{proof}
If $(K,L)\models T_{k,\mathfrak{p}}$, then by definition $(K,L)^\leq\models T_{k,\leq_\mathfrak{p}}$,
and $T_{k,\leq_\mathfrak{p}}$ has quantifier elimination (Theorem \ref{thm:QE_order}). 
So for a $\mathcal{L}_{k,{\rm inf}}$-formula $\psi(x)$,
$\psi_\varphi$ is equivalent modulo $T_{k,\leq_\mathfrak{p}}$ to a quantifier-free formula $\varphi$, and $\varphi_\Psi$ is again quantifier-free (Lemma \ref{lem:interpretation_order}(c)).
Since $(K,L)\models\forall x,y(L(x,y)\leftrightarrow L(x,y)_{\varphi\Psi})$, ${\rm id}_K$ is a homotopy between $\Psi\circ\varphi$ and the identity, and therefore
$\varphi_\Psi$ is equivalent in $(K,L)$ to $\psi$.
We have that $(\underline{k},\mathfrak{p})\models T_{k,\mathfrak{p}}$ since
$(\underline{k},\leq_\mathfrak{p})=(\underline{k},\mathfrak{p})^\leq$ obviously satisfies $T_{k,\leq_\mathfrak{p}}$,
and $(k,\mathfrak{p})\models \forall x,y(L(x,y)\leftrightarrow L(x,y)_{\varphi\Psi})$ by Lemma \ref{lem:real}.
By Proposition \ref{prop:Tk} there is a unique $\mathcal{L}_k$-embedding of $\underline{k}$ into every $(K,L)\models T_{k,\mathfrak{p}}$,
and this embedding is compatible with the ordering by (2) of Definition \ref{def:Tkleq} (namely if $\lambda\geq 0$, then $\lambda\cdot 1\geq\lambda\cdot 0=0$), and therefore also with $L$.
\end{proof}

\subsection{Complex absolute values}

We now discuss complex absolute values.
Recall that an (archimedean) absolute value $|.|$ on $k$ is {\em complex} if the completion $\hat{k}$ of $k$ with respect to $|.|$ is isomorphic to $\mathbb{C}$.
In this case, we get an embedding of $k$ into $\mathbb{C}$ such that $|.|$ becomes the restriction of the usual complex absolute value to $k$.
We denote by $\infty$ the place corresponding to the usual absolute value on $\mathbb{C}$, 
as well as on any subfield of $\mathbb{C}$,
and we also just write $|.|$ for the complex absolute value.

\begin{proposition}\label{prop:complex}
For a subfield $k\subseteq\mathbb{C}$ with $k\not\subseteq\mathbb{R}$,
the subfield $(k\cap\mathbb{R},+,\cdot)$ is definable in $(k,+,{1},L_\infty)$.
\end{proposition}

\begin{proof}
First we claim that the $\mathcal{L}_{k,{\rm inf}}$-formula
$$
 \varphi(x,y) = \forall z(L(z,y)\rightarrow L(x+z,x+y))
$$
defines in $k$ the binary relation $y\in\mathbb{R}_{\geq0}x$.
Indeed, if $y\in\mathbb{R}_{\geq0}x$, then $|x+y|=|x|+|y|$, and every $z\neq y$ with $|z|\leq|y|$ either has $|z|<|y|$, so $|x+z|\leq|x|+|z|<|x|+|y|$, or $z\notin\mathbb{R}_{\geq0}x$, and so $|x+z|<|x|+|z|\leq|x|+|y|$.
On the other hand, if $y\notin\mathbb{R}_{\geq0}x$, then $|x+y|<|x|+|y|$, so since $\mathbb{Q}$ is dense in $k\cap\mathbb{R}$ 
and therefore $\mathbb{Q}x\subseteq k$ is dense in $\mathbb{R}x$,
there exists $z\in k\cap\mathbb{R}x$ with $|x+y|-|x|<|z|<|y|$,
hence $|z|\leq|y|$ but $|x+z|=|x|+|z|>|x+y|$.

In particular, $\varphi(1,z)\vee\varphi(1,-z)$ defines $k\cap\mathbb{R}$.
Next we claim that the $\mathcal{L}_{k,{\rm inf}}$-formula $\psi(a,b,c)$ given by
\begin{eqnarray*}
      &\forall z\forall \epsilon (&(\neg\varphi(1,z)\wedge\neg\varphi(1,-z)\wedge\varphi(1,\epsilon)\wedge \epsilon\neq 0)\rightarrow \\&&\quad\exists x,y(\varphi(z,x)\wedge\varphi(z,y)\wedge\varphi(1+x,b+y)
      \wedge L(a-\epsilon,x,a+\epsilon)\wedge L(c-\epsilon,y,c+\epsilon))\quad)
\end{eqnarray*}
holds in $k$ for $a,b,c\in k\cap\mathbb{R}_{\geq0}$ if and only if $ab=c$. 
Namely, $k\models\psi(a,b,c)$ if and only if for every $z\in k\setminus\mathbb{R}$ and every $\epsilon>0$, there are similar triangles with vertices $(0,1,1+x)$ and $(0,b,b+y)$ with the side between the second and the third vertex parallel to $\mathbb{R}z$ 
and of length $|x|$ respectively $|y|$ (equal to $b|x|$ by similarity), differing from $a$ respectively $c$ by at most $\epsilon$. 
Since $k$ is dense in $\mathbb{C}$, 
and $(\mathbb{C},|.|)$ is complete,
this holds if and only if $c=ab$.
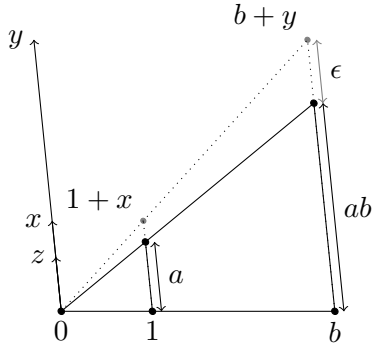
\begin{center}
    \begin{tikzpicture}[scale=1.2,xslant=-0.1]
    \pgfmathsetmacro{\x}{2.3}
    
        \filldraw (0,0) circle (1pt);
        \draw (0,0) node [anchor=north]{$0$};
        \filldraw (1,0) circle (1pt);
        \draw (1,0) node [anchor=north]{$1$};
        \filldraw (3,0) circle (1pt);
        \draw (3,0) node [anchor=north]{$b$};
        
        \filldraw[color=gray] (1,1) circle (0.8pt);
        \draw (1,1) node [anchor=south east]{$1+x$};
        \filldraw[color=gray] (3,3) circle (0.8pt);
        \draw (3,3) node [anchor=south east]{$b+y$};

        \filldraw (1,\x/3) circle (1pt);
        
        \filldraw (3,\x) circle (1pt);

        \draw[dotted] (1,0)--(1,1);
        \draw[dotted] (0,0)--(3,3)--(3,0);
        
        \draw (0,0)--(3,0)--(3,\x)--(0,0);
        \draw (1,0)--(1,\x/3);
        \draw[->] (0,0)--(0,1);
        \draw (0,1) node [anchor=east]{$x$};
        \draw[->] (0,0)--(0,0.6);
        \draw (0,0.6) node [anchor=east]{$z$};
        \draw[->] (0,0)--(0,3);
        \draw (0,3) node [anchor=east]{$y$};
        
        \draw[<->] (3.1,0)--(3.1,\x);
        
        \draw (3.1,\x/2) node [anchor=west]{$ab$};
        \draw[<->] (1.1,0)--(1.1,\x/3);
        \draw (1.1,\x/6) node [anchor=west]{$a$};
        \draw[<->, color=gray] (3.1,\x)--(3.1,3);
        \draw (3.1,1.5+\x/2) node [anchor=west]{$\epsilon$};
    \end{tikzpicture}
\end{center}
From this, one immediately defines multiplication for all $a,b,c\in k\cap\mathbb{R}$ using a case distinction on which of the three elements lie in $k\cap\mathbb{R}_{\geq0}$ (cf.~proof of Proposition \ref{prop:real_undecidable}).
\end{proof}

\begin{corollary}\label{cor:complex_undecidable}
Let $\mathfrak{p}$ be a complex place of a number field $k$.
Then ${\rm Th}(k,+,L_\mathfrak{p})$ is undecidable.
\end{corollary}

\begin{proof}
Use $|.|_\mathfrak{p}$ to embed $k$ into $\mathbb{C}$, and let $k_0:=k\cap\mathbb{R}$. Proposition \ref{prop:complex} then gives that $(k_0,+,\cdot)$ is definable in $(k,+,1,L_\mathfrak{p})$. 
As $k_0$ is a number field, 
${\rm Th}(k_0,+,\cdot)$ is undecidable,
hence ${\rm Th}(k,+,1,L_\mathfrak{p})$ is undecidable,
and since for every $0\neq a\in k$ the map $x\mapsto ax$ is an automorphism of the structure $(k,+,L_\mathfrak{p})$ that maps $1$ to $a$, we can existentially quantify away the constant~$1$.
\end{proof}

We given a strengthening of Proposition \ref{prop:complex} in Appendix \ref{sec:appendix}.
As the full theory is undecidable,
we will now discuss various existential theories.

\begin{lemma}\label{lem:complex_strict}
Let $k\subseteq\mathbb{C}$ be a number field not contained in $\mathbb{R}$.
Then the positive existential theory of $(k,+,1,(\cdot_\lambda)_{\lambda\in k},L')$, where $L'(x,y)\Leftrightarrow|x|<|y|$, is decidable.
\end{lemma}

\begin{proof}
Let $\varphi(z_1,\dots,z_n)$ be a positive quantifier-free formula in the language $(+,1,(\cdot_\lambda)_{\lambda\in k},L')$.
We have to decide whether it is satisfied in $\mathbb{C}$ by some $x\in k^n$.
By case distinctions and introducing new variables,
we can assume without loss of generality that
there are $A\in k^{m\times n}$, $b\in k^m$ and $C\subseteq\{1,\dots,n\}^2$ such that $x\in\mathbb{C}^n$ satisfies $\varphi$ if and only if 
$Ax=b$ and $|x_i|<|x_j|$ for every $(i,j)\in C$.
Note that $U=\{x\in\mathbb{C}^n:|x_i|<|x_j|\mbox{ for }(i,j)\in C\}$
is an open subset of $\mathbb{C}^n$.

Since $k$ is computable, 
we can compute $y_0,\dots,y_d\in k^m$
such that the solution space of $Ax=b$ is $Y:=y_0+\sum_{j=1}^d\mathbb{C}y_j$, and this gives an affine $\mathbb{C}$-linear isomorphism
$f\colon \mathbb{C}^d\rightarrow Y$, $x\mapsto y_0+\sum_jx_jy_j$.
Note that $f(k^d)=Y\cap k^n$,
so $\varphi$ is satisfiable by an element of $k^n$ if and only if 
$f^{-1}(U)\cap k^d\neq\emptyset$.
Since $f^{-1}(U)$ is again open 
and $k$ is dense in $\mathbb{C}$,
this is equivalent to $f^{-1}(U)\neq\emptyset$,
equivalently $U\cap Y\neq\emptyset$.

That is, $\varphi$ is satisfiable by an element of $k^n$ if and only it is satisfiable in $\mathbb{C}$,
and since 
$(\mathbb{C},+,(\cdot_\lambda)_{\lambda\in k},L')$
is interpretable in the decidable field
$(\mathbb{R},+,\cdot)$,
this can be decided.
\end{proof}

\begin{proposition}\label{prop:complex_existential_decidable}
Let $k\subseteq\mathbb{C}$ be a number field.
If $k\cap \mathbb{R}i=\{0\}$, then the existential theory of
$(k,+,1,(\cdot_\lambda)_{\lambda\in k},L_\infty)$ is decidable.
\end{proposition}

\begin{proof}
Note that $|.|$ induces an embedding $k^\times/\{\pm1\}\rightarrow\mathbb{R}^\times$. 
Indeed, if $|z|=|z'|$ and $z\neq z'$, then
$\frac{z+z'}{z-z'}=\frac{|z|^2-|z'|^2+\bar{z}{z}'-{z}\bar{z}'}{|z-z'|^2}=\frac{2i{\rm Im}(z\bar{z}')}{|z-z'|^2}\in k\cap \mathbb{R}i=\{0\}$, hence $z+z'=0$.

Now let $\varphi(z_1,\dots,z_n)$ be a quantifier-free $\mathcal{L}_{k,{\rm inf}}$-formula.
Without loss of generality, all negations in $\varphi$ are in front of atomic formulas.
Using that by the previous paragraph, 
\begin{eqnarray*}
    |x|\leq|y| &\Leftrightarrow& x=y\vee x=-y\vee|x|<|y|\\
   \neg(|x|\leq|y|) &\Leftrightarrow& |y|<|x|\\
   \neg(x=y) &\Leftrightarrow& x=-y\vee|x|<|y|\vee |y|<|x|
\end{eqnarray*}
for all $x,y\in k$,
we see that we can compute from $\varphi$ a positive quantifier-free formula $\varphi'$
in the language $(+,1,(\cdot_\lambda)_{\lambda\in k},L')$ 
that is equivalent to $\varphi$ in
$(k,+,1,(\cdot_\lambda)_{\lambda\in k},L,L')$,
where $L'$ is interpreted as in Lemma \ref{lem:complex_strict},
which also allows us to decide whether $\varphi'$ is satisfiable.
\end{proof}

\begin{proposition}\label{prop:k0exdef}
Let $k\subseteq\mathbb{C}$ be a subfield and let $k_0=k\cap\mathbb{R}$.
If $k\cap \mathbb{R}i\neq\{0\}$, then the field
$(k_0,+,\cdot)$ is existentially definable in
$(k,+,1,(\cdot_\lambda)_{\lambda\in k_0},L_\infty)$.
\end{proposition}

\begin{proof}
First note that the quantifier $\mathcal{L}_{k_0,{\rm inf}}$-formula
$$
 \iota(x) :=  \Lambda(x-1,x+1) 
$$ 
defines $k \cap \mathbb{R}i$ in $k$, and by assumption there exists
$0\neq\alpha\in k\cap\mathbb{R}i$. 
Therefore, the existential $\mathcal{L}_{k_0,{\rm inf}}$-formula
$$
 \rho(x) := \exists z(z\neq 0 \wedge \iota(z) \wedge \Lambda(x-z,x+z)) 
$$
defines $k_0$ in $k$.
Let $a=-\alpha^2$
and
observe that the quantifier-free $\mathcal{L}_{k,{\rm inf}}$-formula  
$$
 \varphi(x,y,z) :=\Lambda(2x+\alpha(y-a^{-1}z),\alpha(y+a^{-1}z))
$$
holds for $x,y,z \in k_0$ if and only if
$4x^2+a(y-a^{-1}z)^2=a(y+a^{-1}z)^2$,
i.e.~$x^2=yz$.
Thus, 
$$
  \exists u,v(\varphi(y,1,u)\wedge\varphi(z,1,v)\wedge\varphi(x,u,v))
$$
holds if and only if $x^2=y^2z^2$, 
i.e.~$|x|=|yz|$,
hence the claim follows by using Proposition~\ref{prop:real_undecidable} (applied to $k_0$).

It remains to note that the scalar multiplication by $\alpha$ can be replaced with scalar multiplication by $a$,
since the relation $y= \pm x\alpha $ on $k_0\times k_0\alpha$ is definable by the quantifier-free $\mathcal{L}_{k_0,{\rm inf}}$-formula
    \[
     \Lambda\left((a+1)x+2y, (a-1)x \right ).
     \] 
Indeed, the formula above holds for $x\in k_0$ and $y\in k_0\alpha$ if and only if $(a+1)^2x^2-4y^2 = (a-1)^2x^2$,
hence if and only if $4ax^2=4y^2$.
\end{proof}

\begin{corollary}\label{cor:complex_Turing_equiv}
Let $k\subseteq\mathbb{C}$ be a number field and let $k_0=k\cap\mathbb{R}$.
If $k\cap \mathbb{R}i\neq\{0\}$, then 
${\rm Th}_\exists(k,+,(\cdot_\lambda)_{\lambda\in k_0},L_\infty)$,
${\rm Th}_\exists(k,+,1,(\cdot_\lambda)_{\lambda\in k},L_\infty)$,
and ${\rm Th}_\exists(k_0,+,\cdot)$ 
are Turing-equivalent.
\end{corollary}

\begin{proof}
Proposition \ref{prop:k0exdef} immediately gives
a Turing reduction of  $T_0:={\rm Th}_\exists(k_0,+,\cdot)$ to
$T_1:={\rm Th}_\exists(k,+,1,(\cdot_\lambda)_{\lambda\in k_0},L_\infty)$,
and since the map $x\mapsto ax$ for $0\neq a\in k$ is an automorphism of
$(k,+,(\cdot_\lambda)_{\lambda\in k_0},L_\infty)$ that sends $1$ to $a$,
we can existentially quantify away the constant $1$.

Conversely, $(k,+,\cdot,(\lambda)_{\lambda\in k},L_\infty)$
is existentially interpretable in $(k_0,+,\cdot,(\lambda)_{\lambda\in k_0})$ since $k/k_0$ is finite,
and so since each $\lambda\in k_0$ is existentially $\emptyset$-definable in the field $k_0$ up to automorphisms of $k_0$,
$T_1$
can be reduced to
$T_0$
by existentially quantifying away the constants.
\end{proof}

\section{Weak approximation}
\label{sec:WA}

\noindent
In this final section, 
we will axiomatize and make use of the well-known weak approximation theorem, to study the theory of a field with a family of absolute values.

\begin{proposition}[Artin--Whaples, \cite{AW}]\label{prop:WA}
Let $|.|_1,\dots,|.|_r$ be pairwise inequivalent non-trivial absolute values on a field $k$,
and let $x_1,\dots,x_r\in k$.
Then for every $\epsilon>0$ there exists $y\in k$ such that $|y-x_i|_i<\epsilon$ for every $i$.
\end{proposition}

This following immediate corollary is a straightforward generalization of \cite[Lemma 29]{BF}.

\begin{corollary}\label{cor:WA}
Let $\mathfrak{p}_1,\dots,\mathfrak{p}_r\in\mathbb{P}_k$ be pairwise distinct.
Let $A\in k^{m\times n}$, $b\in k^m$ and $x_1,\dots,x_r\in k^n$ with $Ax_i=b$ for $i=1,\dots,r$.    
Then for every $\epsilon>0$ there exists $y\in k^n$ with $Ay=b$ and $|y_j-x_{i,j}|_{\mathfrak{p}_i}<\epsilon$ for every $i=1,\dots,r$ and $j=1,\dots,n$.
\end{corollary}

\begin{proof}
Let $Y=\{y\in k^n:Ay=b\}$ and fix an affine $k$-linear isomorphism $f\colon Y\rightarrow k^d$ (for some $d\leq n$).
Note that $f$ is a homeomorphism with respect to the topology induced by $\mathfrak{p}_i$, for each $i$, and therefore we can assume without loss of generality that $f={\rm id}$, so $A=0$, $b=0$, $n=d$.
The claim then follows by applying Proposition \ref{prop:WA} to each coordinate $j=1,\dots,n$ separately.
\end{proof}

\begin{definition}
Let $S_1\subseteq S_0\subseteq\mathbb{P}_k$ be sets of places of $k$.  
We let 
$$
 \mathcal{L}_{k,S_0,S_1}=\mathcal{L}_k\cup\bigcup_{\mathfrak{p}\in S_0}\mathcal{L}_{k,\mathfrak{p}}\cup \bigcup_{\mathfrak{p}\in S_1}\mathcal{L}_{k,\mathfrak{p}}^+
$$
where $\mathcal{L}_{k,\mathfrak{p}}^{[+]}$ is the language $\mathcal{L}_{k,{\rm fin}}^{[+]}$ (if $\mathfrak{p}$ is finite) respectively $\mathcal{L}_{k,{\rm inf}}^{[+]}$ (if $\mathfrak{p}$ is infinite) with all symbols that are not in $\mathcal{L}_k$ labeled by $\mathfrak{p}$ (e.g. $L^\mathfrak{p}$ instead of $L$, so that $\mathcal{L}_{k,\mathfrak{p}}^{[+]}\cap\mathcal{L}_{k,\mathfrak{q}}^{[+]}=\mathcal{L}_k$ for $\mathfrak{q}\neq\mathfrak{p}$).
We write $(\underline{k},S_0,S_1)$ for $k$ viewed as an $\mathcal{L}_{k,S_0,S_1}$-structure in the obvious way.
\end{definition}

\begin{definition}\label{def:TS0S1}
Let $S_1\subseteq S_0\subseteq\mathbb{P}_k'$,
where $S_0$ contains no complex places and
$S_1$ contains no infinite places.
We let
$$
 T_{k,S_0,S_1} = T_k\cup\bigcup_{\mathfrak{p}\in S_0}T_{k,\mathfrak{p}}\cup \bigcup_{\mathfrak{p}\in S_1}T_{k,\mathfrak{p}}^+,
$$
where we view $T_{k,\mathfrak{p}}^{[+]}$ as an $\mathcal{L}_{k,\mathfrak{p}}^{[+]}$-theory, together with the axiom scheme ({\em weak approximation}) for a structure $K$,
which contains for every finite subset $S=\{\mathfrak{p}_1,\dots,\mathfrak{p}_r\}\subseteq S_0$ and every matrix $A\in k^{m\times n}$ the following axiom: For all $b\in K^m$, $c_1,\dots,c_r\in K\setminus\{0\}$, and $x_1,\dots,x_r\in K^n$ such that $Ax_i=b$ for $i=1,\dots,r$, there exists $y\in K^n$ with $Ay=b$ and $L^{\mathfrak{p}_i}(y_j-x_{i,j},c_i)$ for every $i=1,\dots,r$ and every $j=1,\dots,n$.
\end{definition}

\begin{proposition}\label{prop:k_S}
Let $S_1\subseteq S_0\subseteq\mathbb{P}_k'$ be as in Definition \ref{def:TS0S1}. 
Then $(\underline{k},S_0,S_1)\models T_{k,S_0,S_1}$ and $(\underline{k},S_0,S_1)$ embeds uniquely into every model of $T_{k,S_0,S_1}$.     
\end{proposition}

\begin{proof}
The fact that $(\underline{k},S_0,S_1)$ satisfies each of the $T_{k,\mathfrak{p}}^+$ is part of Theorems \ref{thm:finite} and \ref{thm:real},
and it satisfies
the axiom scheme weak approximation by Corollary \ref{cor:WA}.
If $K\models T_{k,S_0,S_1}$,
then there is a unique embedding of $k$ into $K$ as an $\mathcal{L}_k$-structure (Proposition \ref{prop:Tk}), 
and this embedding is also an
$\mathfrak{L}_{k,\mathfrak{p}}^{[+]}$-embedding for every $\mathfrak{p}\in S_0$ (also see Theorems \ref{thm:finite} and \ref{thm:real}), and therefore a $\mathcal{L}_{k,S_0,S_1}$-embedding.
\end{proof}

\begin{theorem}\label{thm:WA}
Let $S_1\subseteq S_0\subseteq\mathbb{P}_k'$.
\begin{enumerate}[(a)]
\item If $S_0$ contains no complex place and $S_1$ contains no infinite place, then $T_{k,S_0,S_1}$ is model complete and complete.
\item If $S_0$ contains no complex place, and $S_1=\emptyset$, then $T_{k,S_0,S_1}$ has quantifier elimination.
\end{enumerate}
\end{theorem}

\begin{proof}
We first prove (a).
By Proposition \ref{prop:k_S}, completeness follows from model completeness, so this is what we aim to prove.
Let $K_0\subseteq K_1$ be models of $T_{k,S_0,S_1}$ and let $\psi(z_1,\dots,z_n)$ be a quantifier-free $\mathcal{L}_{k,S_0,S_1}(K_0)$-formula that is satisfied in $K_1$, say $K_1\models\psi(a)$, $a\in K_1^n$. 
We want to show that $\psi$ is satisfied also in $K_0$.

By case distinctions and introducing new variables,
we can assume without loss of generality that 
$\psi$ is a conjunction of atomic and negated atomic formulas, and we can write 
$$
 \psi=\psi_0\wedge \psi_1\wedge\dots\wedge\psi_r
$$
where 
\begin{itemize}
\item  $\psi_0$ is a conjunction of atomic and negated atomic $\mathcal{L}_k(K_0)$-formulas, and
\item $\psi_i$ for $i=1,\dots,r$ is a conjunction of atomic and negated atomic formulas in the language
$\mathcal{L}_i:=\mathcal{L}_{k,\mathfrak{p}_i}^{[+]}\setminus\mathcal{L}_k$ not involving equality,
\end{itemize} 
with $\mathfrak{p}_1,\dots,\mathfrak{p}_r\in S_0$ distinct.
By allowing $\psi_i$ to be formulas in the bigger language
$\mathcal{L}_i\cup\{0,\cdot_{-1}\}$,
we can assume without loss of generality
that $a_j\neq0$ for each $j$
(if $a_j=0$, replace each occurrence of $z_j$ in $\psi$ by the symbol $0$) 
and $a_{j_1}\neq \pm a_{j_2}$ for each $j_1\neq j_2$
(if $a_{j_1}=a_{j_2}$, replace each occurrence of $z_{j_2}$ by $z_{j_1}$, and if
$a_{j_1}=-a_{j_2}$ replace each occurrence of $z_{j_2}$ by $-1\cdot z_{j_1}$).
Then, without loss of generality
there exist $A\in k^{m\times n}$, $b\in K_0^n$ such that for any $K\models T_k$ and $x\in K^m$, $\psi_0(x)$ holds if and only if $x\in (K\setminus \{0\})^n$, $x_{j_1}\neq \pm x_{j_2}$ for each $j_1\neq j_2$, and $Ax=b$.

Now for each $i$ we have that
the $\mathcal{L}_{k,\mathfrak{p}_i}^{[+]}(K_0)$-formula $\psi_0\wedge\psi_i$ is satisfied in $K_1$.
As $K_0,K_1\models T_{k,\mathfrak{p}_i}^{[+]}$,
and $T_{k,\mathfrak{p}_i}^{[+]}$ is model complete (Theorems \ref{thm:finite}, \ref{thm:real}), we get that
$\psi_0\wedge\psi_i$ is satisfied in $K_0$.
That is, there exists $x_i\in (K_0\setminus\{0\})^n$ with $x_{i,j_1}\neq \pm x_{i,j_2}$ for each $j_1\neq j_2$, $Ax=b$ and $K_0\models\psi_i(x_i)$.
We now apply the axiom scheme {\em (weak approximation)} from Definition \ref{def:TS0S1} to obtain $y\in K_0^n$ with $Ay=b$
and $L_{\mathfrak{p}_i}(y_j-x_{i,j},c_i)$ for every $i,j$, with $c_i\in K_0\setminus\{0\}$ that we now only need to choose in order to ensure that $K_0\models\psi_i(y)$ for each $i$, $y_j\neq 0$ for each $j$, and $y_{j_1}\neq \pm y_{j_2}$ for each $j_1\neq j_2$.

If $\mathfrak{p}_i$ is finite, then  $K_0\models T_{k,\mathfrak{p}_i}^{[+]}$
implies that
$K_0$ interprets some $(K_0,\Gamma,v)=K_0^{(2)}\models T_{k,{\rm val}}^{[+]}$
(Lemma~\ref{lem:interpretation}). 
As $x_{i,j}\neq 0$ for each $j$, and $x_{i,j_1}\neq\pm x_{i,j_2}$ for each $j_1\neq j_2$, there exists $c_i\in K_0\setminus\{0\}$ with $v(c_i)>v(x_{i,j})$ for each $j$ and with $v(c_i)>v(x_{i,j_1}\pm x_{i,j_2})$ for each $j_1\neq j_2$, and we claim that this $c_i$ does the job.
Indeed, $L_{\mathfrak{p}_i}(y_j-x_{i,j},c_i)$ means that $v(y_j-x_{i,j})\geq v(c_i)$, hence $v(y_j-x_{i,j})>v({x_{i,j}})$ and therefore $v(y_j)=v(x_{i,j})$.
So since the truth of $K_0\models L^{\mathfrak{p}_i}(\pm z_{j_1},\pm z_{j_2})$,
$K_0\models M^{\mathfrak{p}_i}(\pm z_{j_1},\pm z_{j_2},\pm z_{j_3})$ and
$K_0\models Q_{l}^{\mathfrak{p}_i}(\pm z_j)$ 
depends only on $v(z_{j})$ for each $j$ (and similarly if one of the $z_j$ is replaced by the constant symbol $0$), 
we deduce from $K_0\models \psi_i(x_i)$ that also $K_0\models\psi_i(y)$.
Note that $v(y_j)=v(x_{i,j})$ implies that $y_j\neq 0$;
and $v(y_{j_1}\pm y_{j_2})=v((y_{j_1}-x_{i,j_1})\pm (y_{j_2}-x_{i,j_2})+(x_{i,j_1}\pm x_{i,j_2}))$, which since $v(y_{j_1}-x_{i,j_1}),v(y_{j_2}-x_{i,j_2})\geq v(c_i)>v(x_{i,j_1}\pm x_{i,j_2})$ implies that $y_{j_1}\neq \pm y_{j_2}$ for every $j_1\neq j_2$.

If $\mathfrak{p}_i$ is real, then similarly 
$K_0$ interprets some $(K_0,\leq)\models T_{k,{\rm ord}}$.
As $x_{i,j}\neq 0$ for each $j$ and $x_{i,j_1}\neq\pm x_{i,j_2}$ for each $j_1\neq j_2$, there exists 
$c_i\in K_0$ with $c_i>0$,
$c_i<\frac{1}{2}|x_{i,j}|$ for each $j$, 
and $c_i<\frac{1}{2}|x_{i,j_1}-x_{i,j_2}|$ for each $j_1\neq j_2$, and we claim that this does the job.
Indeed, note that
$\psi_i$ is equivalent to a conjunction of formulas of the form $L^{\mathfrak{p}_i}(t,s)$, where $t,s\in\{0,z_1,\dots,z_n,-z_1,\dots,-z_n\}$, and negations of such formulas.
Suppose that $K_0\models L^{\mathfrak{p}_i}(\pm x_{i,j_1},\pm x_{i,j_2})$,
which is equivalent to
$|x_{i,j_2}|\geq|x_{i,j_1}|$.
Let us assume that $x_{i,j_2}\geq x_{i,j_1}\geq 0$, the other cases are similar. 
The case $j_1=j_2$ is trivial, so assume that $j_1\neq j_2$.
Then $0<2c_i<x_{i,j_2}-x_{i,j_1}$.
For each $j$, $K_0\models L^{\mathfrak{p}_i}(y_j-x_{i,j},c_i)$ implies that $|y_j-x_{i,j}|\leq c_i$, so
$$
 y_{j_2}-y_{j_1}\geq (x_{i,j_2}-c_i)-(x_{i,j_1}+c_i)=(x_{i,j_2}-x_{i,j_1})-2c_i>0,
$$ 
in particular $K_0\models L^{\mathfrak{p}_i}(y_{j_1},y_{j_2})$.
And if $K_0\models\neg L^{\mathfrak{p}_i}(\pm x_{i,j_1},\pm x_{i,j_2})$, then $K_0\models L^{\mathfrak{p}_i}(\pm x_{i,j_2},\pm x_{i,j_1})$, and the previous argument shows that
$K_0\models L^{\mathfrak{p}_i}(y_{j_2},y_{j_1})\wedge y_{j_2}\neq y_{j_1}$
and thus
$K_0\models\neg L^{\mathfrak{p}_i}(y_{j_1},y_{j_2})$.
Therefore, we deduce from $K_0\models \psi_i(x_i)$ that
$K_0\models\psi_i(y)$.
We have also already shown that $y_{j_1}\neq y_{j_2}$ for $j_1\neq j_2$,
and $y_j\neq 0$ follows similarly from $|x_{i,j}|>2c_i$ and $|y_j-x_{i,j}|\leq c_i$.

The proof of (b) now works very similar, except that we start with two models $K_1,K_2$ of $T_{k,S_0,S_1}$ with a common sub{\em structure} $K_0$, and we want to conclude from $K_1\models\exists{z}\psi$ that
$K_2\models\exists{z}\psi$.
We proceed as before but then use that $T_{k,\mathfrak{p}_i}$ has quantifier elimination to conclude from $K_1\models\exists z(\psi_0\wedge\psi_i)$ that $K_2\models\exists z(\psi_0\wedge\psi_i)$. 
The same weak approximation argument as before then shows that $K_2\models\exists z\psi$.
\end{proof}

\begin{corollary}\label{Cor:global_decidable}
If $k$ is a global field
and $S_1\subseteq S_0\subseteq\mathbb{P}_k$ are sets of places of $k$ where $S_0$ contains no complex places and $S_1$ contains no infinite places, then $T_{k,S_0,S_1}$ is decidable.
\end{corollary}

\begin{proof}
If $k$ is a global field, then $(k,+,\cdot)$ is computable (see, e.g.~\cite[\S19.2]{FJ}), and also each of the relations 
$L_{\mathfrak{p}},M_{\mathfrak{p}},Q_\mathfrak{p}$ is computable.
Therefore, both the theory $T_k$ and each of the theories $T_{k,\mathfrak{p}}^{[+]}$ is computable.
Hence also $T_{k,S_0,S_1}$ is computable, and therefore the set of consequences of 
$T_{k,S_0,S_1}$ is computably enumerable.
As $T_{k,S_0,S_1}$ is complete (Theorem \ref{thm:WA}), this implies its decidability.
\end{proof}

Note that we have not defined $T_{k,S_0,S_1}$ when $S_0$ contains a complex place or  $S_1$ contains a real place.  
Finally, we put everything together to prove the theorems from the introduction:

\begin{proof}[Proof of Theorem \ref{thm:intro1}]
Let $k$ be a global field and $\mathfrak{p}\in\mathbb{P}_k=\mathbb{P}_k'$.

(a): Suppose that $\mathfrak{p}$ is finite.
By Theorem \ref{thm:finite},
$$
 (k,+,0,1,(\cdot_\lambda)_{\lambda\in k},L_\mathfrak{p},M_\mathfrak{p},(Q_{n,\mathfrak{p}})_{n\in\mathbb{N}})=(\underline{k},\mathfrak{p})\models T_{k,\mathfrak{p}}^+.
$$ 
Similarly, 
$$
 (\hat{\underline{k}}_\mathfrak{p},\hat{\mathfrak{p}}) := (\hat{k}_\mathfrak{p},+,0,1,(\cdot_\lambda)_{\lambda\in k},L_{\hat{\mathfrak{p}}},M_{\hat{\mathfrak{p}}},(Q_{n,{\hat{\mathfrak{p}}}})_{n\in\mathbb{N}})\models T_{k,\mathfrak{p}}^+.
$$ 
Indeed, $(\hat{\underline{k}}_\mathfrak{p},\hat{\mathfrak{p}})^{(2)}
=(\underline{\hat{k}}_{\mathfrak{p}},\mathbb{Z}\cup\{\infty\},v_{\hat{\mathfrak{p}}})$, 
where $v_{\hat{\mathfrak{p}}}$ is the natural extension of $v_\mathfrak{\mathfrak{p}}$ to the completion,
and one easily checks that $(\underline{\hat{k}}_{\mathfrak{p}},\mathbb{Z}\cup\{\infty\},v_{\hat{\mathfrak{p}}})\models T_{k,v_\mathfrak{p}}^+$.
Since $T_{k,\mathfrak{p}}^+$ is complete by Theorem \ref{thm:finite},
and decidable by Corollary \ref{Cor:global_decidable} (for $S_0=\{\mathfrak{p}\}$, $S_1=\emptyset$),
we obtain that also the $\mathcal{L}_{k}\cup\{L,M\}$-reducts of
$(\underline{k},\mathfrak{p})$ and
$(\hat{\underline{k}}_\mathfrak{p},\hat{\mathfrak{p}})$
have the same theory,
and that this theory is decidable.

(b): Suppose that $\mathfrak{p}$ is real.
By Theorem \ref{thm:real}, $(\underline{k},L_\mathfrak{p})\models T_{k,\mathfrak{p}}$,
and similarly
$(\hat{\underline{k}}_\mathfrak{p},L_{\hat{\mathfrak{p}}})\models T_{k,\mathfrak{p}}$.
So as $T_{k,\mathfrak{p}}$ is complete (Theorem \ref{thm:real}), we get that
${\rm Th}(\underline{k},L_\mathfrak{p})={\rm Th}(\hat{\underline{k}}_\mathfrak{p},L_{\hat{\mathfrak{p}}})$,
and this theory is decidable 
 by Corollary \ref{Cor:global_decidable} (again for $S_0=\{\mathfrak{p}\}$, $S_1=\emptyset$).
By Corollary \ref{cor:real_undecidable},
${\rm Th}(\underline{k},L_\mathfrak{p},M_\mathfrak{p})$ is undecidable,
so since $(\mathbb{R},+,0,1,(\cdot_\lambda)_{\lambda\in k},L_\infty,M_\infty)$ is interpretable in the decidable field
$(\mathbb{R},+,\cdot)$ (note that every $\lambda\in k$ is definable in $\mathbb{R}$), we get in particular that
$$
 {\rm Th}(\underline{k},L_\mathfrak{p},M_\mathfrak{p}) \neq {\rm Th}(\underline{\hat{k}}_\mathfrak{p},L_{\hat{\mathfrak{p}}},M_{\hat{\mathfrak{p}}}).
$$
The statement about the existential theory is
contained in Corollary \ref{cor:real_undecidable}.

(c): Suppose that $\mathfrak{p}$ is complex.
By Corollary \ref{cor:complex_undecidable},
${\rm Th}(k,+,L_\mathfrak{p})$ is undecidable. 
Since $(\mathbb{C},+,L_\infty)$ is interpretable in 
the decidable field $(\mathbb{R},+,\cdot)$, we get that
in particular
${\rm Th}(k,+,L_\mathfrak{p}) \neq {\rm Th}(\hat{k}_\mathfrak{p},+,L_{\hat{\mathfrak{p}}})$.
If $k\cap \mathbb{R}i=\{0\}$,
then
${\rm Th}_\exists(k,+,1,(\cdot_\lambda)_{\lambda\in k},L_{\mathfrak{p}})$
is decidable by Proposition \ref{prop:complex_existential_decidable}.
If $k\cap \mathbb{R}i\neq\{0\}$,
then ${\rm Th}_\exists(k,+,1,(\cdot_\lambda)_{\lambda\in k},L_{\mathfrak{p}})$ is Turing-equivalent to
${\rm Th}_\exists(k\cap\mathbb{R},+,\cdot)$ by Corollary \ref{cor:complex_Turing_equiv}.
\end{proof}

\begin{proof}[Proof of Theorem \ref{thm:intro2}]
Let $k$ be a global field and $S_1\subseteq S_0\subseteq\mathbb{P}_k=\mathbb{P}_k'$.
If $S_0$ contains a complex place or
$S_1$ contains a real place, 
the undecidability follows from Theorem \ref{thm:intro1}.
If neither of this happens, 
we have that $(\underline{k},S_0,S_1)\models T_{k,S_0,S_1}$ by Proposition \ref{prop:k_S},
and
$T_{k,S_0,S_1}$ is complete by Theorem \ref{thm:WA}
and decidable by
Corollary \ref{Cor:global_decidable},
and thus ${\rm Th}(\underline{k},S_0,S_1)$ is decidable.
\end{proof}

\appendix
\section{}
\label{sec:appendix}
\noindent
We saw in \Cref{prop:complex} that for any subfield $k$ of $\mathbb{C}$ which is not contained in $\mathbb{\mathbb{R}}$, the field $(k\cap \mathbb{R},+,\cdot)$ is definable in $(k,+,0,1,L_\infty)$. 
    We now extend this to show that the full field structure $(k,+,\cdot)$ is definable. 
    
    \begin{proposition}\label{prop:complex+}
        For a subfield $k$ of $\mathbb{C}$ with $k \not\subseteq \mathbb{R}$, the field $(k,+,\cdot)$ is definable in $(k,+, 0, 1,L_\infty)$.
    \end{proposition}

\begin{proof}
As in the proof of \Cref{prop:complex}, let $\varphi(x,y)$ be the formula that defines the relation $y\in \mathbb{R}_{\geq 0}x$,
and $\psi(r,r',r'')$ the formula that defines the ternary relation $rr'=r''$ on $k_0:=k\cap\mathbb{R}$.
In particular, the formula  $\varphi(1,x)$ defines $k^+\coloneqq k \cap \mathbb{R}_{\geq 0}$.

We first define the relation $rz'=z''$ on $k^+ \times k\times k$, i.e.~the multiplication by a positive scalar~$r$. 
We start with a formula $\theta(a,z,\epsilon)$ stating, for elements $a,z,\epsilon \in k$, that $a,\epsilon \in k^+$, $\epsilon\leq a$ and $ a -\epsilon \leq \vert z \vert \leq a +\epsilon$.
We may take: 
\begin{eqnarray*}
 \theta(a,z,\epsilon) &\coloneqq&\varphi(1,a)\wedge \varphi(1,\epsilon) \wedge L(\epsilon,a) \wedge L(a-\epsilon, z, a+\epsilon)   
\end{eqnarray*}
Now, for $r \in k^+$ and $z,w\in k$, we have $rz = w$ if and only if the following formula $\psi'(r,z,w)$ holds:
\[
 \varphi(w,z) \wedge \forall \epsilon,   \nu,  r',  r''  (( \theta(r',z,\epsilon) \wedge \varphi(1,r') \wedge \varphi(1,r'') \wedge \psi(r,r',r'')  \wedge \psi(r,\epsilon,\nu)) \rightarrow  \theta(r'',w,\nu) ). 
\]
Indeed, if for all $\epsilon$ and $r'\in k^+$ with  $||z|-r' | < \epsilon$ we have
$||w|- rr' |<r\epsilon$, then $|rz| = |w|$. If moreover $w\in \mathbb{R}_{\geq0} z$, then $w=rz$, as wanted.

Then, for fixed $\delta \in k^+$ the formula
\begin{eqnarray*}
 \chi(\delta,w) &:=& L(w-\delta+1, w-\delta-1) \wedge 
 L(w+\delta-1,w+\delta+1)\wedge\theta(1,w,\delta)   
\end{eqnarray*}
defines the predicate $|{\rm Re}(w)|\leq \delta \wedge ||w|-1|\leq \delta$, 
which describes a set $U_\delta\cup\overline{U_\delta}$, where $U_\delta$ is a neighborhood of $i$
of diameter $\rho\rightarrow 0$ as $\delta\rightarrow 0$.
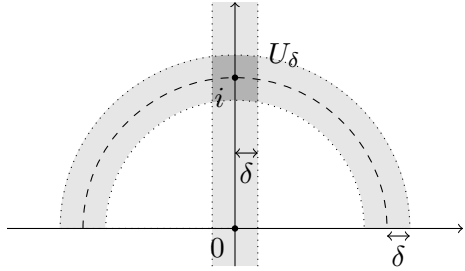
\begin{center}
    \begin{tikzpicture}    \pgfmathsetmacro{\d}{0.3}
        \filldraw[black!10] (0,0) -- (180:2+\d) arc (180:0:2+\d);
        \filldraw[white] (0,0) -- (180:2-\d) arc (180:0:2-\d); 
        \fill[black!10] (-\d,-0.5) rectangle (+\d,3);

\begin{scope}[even odd rule]
  \clip
    (0,0) circle (2+\d)
    (0,0) circle (2-\d);
    \fill[black!30] (-\d,0) rectangle (+\d,3);
\end{scope}

        \draw[dotted] (0,0) -- (180:2+\d) arc (180:0:2+\d);
        \draw[dotted] (0,0) -- (180:2-\d) arc (180:0:2-\d);
        \draw[dotted] (\d,-0.5) -- (\d,3);
        \draw[dotted] (-\d,-0.5) -- (-\d,3);

        \draw[<->] (0,1)--(\d,1);
        \draw[<->] (2,-0.1)--(2+\d,-0.1);
        \draw (\d/2,1) node[anchor=north] {$\delta$};
        \draw (2+\d/2,-0.1) node[anchor=north] {$\delta$};

        \draw [dashed](180:2) arc (180:0:2);
        \draw[->] (-3,0)--(3,0);
        \draw[->] (0,-0.5)--(0,3);
        \draw (0,0) node[anchor=north east] {$0$};
        \filldraw (0,0) circle (1pt);
        \draw (0,2) node[anchor=north east] {$i$};
        \filldraw (0,2) circle (1pt);
        \draw (\d,2) node[anchor=south west] {$U_\delta$};
        
    \end{tikzpicture}
\end{center}

\vspace{-2cm}

With $D:=\mathbb{R}^4\times(\mathbb{C}\setminus\mathbb{R})$,
the map
$$
 f\colon\begin{cases}\quad\quad D&\longrightarrow\quad\mathbb{C}^2\times(\mathbb{C}\setminus\mathbb{R}) \\ (a,b,a',b',w)&\longmapsto\quad (a+bw,a'+b'w,w)
 \end{cases}
$$
is a homeomorphism, and
$$
g\colon\begin{cases}
\quad\quad D&\longrightarrow\quad \mathbb{C}\\(a,b,a',b',w)&\longmapsto\quad aa'-bb'+(ab'+a'b)w
\end{cases}
$$
is continuous.
Thus if 
$d_n\in D$ is a sequence such that
$f(d_n)$ converges to $(z,z',i)$,
then $g(d_n)$ converges to 
$$
 g(f^{-1}(z,z',i))=g({\rm Re}(z),{\rm Im}(z),{\rm Re}(z'),{\rm Im}(z'),i) = zz',
$$
and if $f(d_n)$ converges to $(z,z',-i)$, then
$g(d_n)$ converges to
$$
 g(f^{-1}(z,z',-i))=g({\rm Re}(z),-{\rm Im}(z),{\rm Re}(z'),-{\rm Im}(z'),-i) = zz'
$$
as well.
Therefore, since $k_0$ is dense in $\mathbb{R}$, and $k_0+k_0w$ is dense in $\mathbb{C}$ for any $w\in k\setminus k_0$,
\begin{eqnarray*}
&&\forall 0<\epsilon\in k_0\, \exists 0<\delta\in k_0\, \forall a,b,a',b'\in k_0\, \forall w\in k\setminus k_0\\
&&\quad((w\in U_\delta \wedge |(a+bw)-z|<\delta \wedge |(a'+b'w)-z'|<\delta )\rightarrow |g(a,b,a',b',w)-z''|<\epsilon)
\end{eqnarray*}
holds for $z,z',z''\in k$ if and only if $zz'=z''$,
and using the formulas $\psi'$, $\chi$ and $\varphi$ from above, this can be expressed by a formula in the language $\{+,0,1,L\}$.
\end{proof}

\section*{Acknowledgements}
\noindent
The authors would like to thank Guang Hu for 
help with the arguments in Propositions
\ref{prop:complex_existential_decidable}, \ref{prop:k0exdef} and \ref{prop:complex+},
and Philip Dittmann for helpful remarks on a previous version.
A.~F.\ would like to thank Manuel Bodirsky for earlier helpful discussions regarding the linear theory of the $p$-adic numbers and \cite{Weispfenning},
and Konstantinos Kartas for sharing an unpublished note in which he proves a strengthening of Theorem \ref{cor:QE}.

Most of this research was done while the authors were participating in the trimester on ``Definability, decidability, and computability'' at the Hausdorff Institute Bonn, funded by the Deutsche Forschungsgemeinschaft (DFG, German Research Foundation) under Germany‘s Excellence Strategy – EXC-2047/1 – 390685813.

\end{document}